\documentclass{article}

\usepackage{arxiv}

\usepackage[utf8]{inputenc} % allow utf-8 input
\usepackage[T1]{fontenc}    % use 8-bit T1 fonts
\usepackage{hyperref}       % hyperlinks
\usepackage{url}            % simple URL typesetting
\usepackage{booktabs}       % professional-quality tables
\usepackage{amsfonts}       % blackboard math symbols
\usepackage{nicefrac}       % compact symbols for 1/2, etc.
\usepackage{microtype}      % microtypography
\usepackage{lipsum}		% Can be removed after putting your text content
\usepackage{graphicx}
\usepackage{mathtools}
\usepackage{amssymb}
\usepackage{amsthm}
\usepackage{amsmath}
\usepackage{doi}

\usepackage{cases}
\usepackage{microtype}
\usepackage[utf8]{inputenx}
\usepackage{amssymb}
\usepackage{amsthm}
\usepackage{mathtools, cases}
\usepackage{mathbbol}
\usepackage{makecell}
\usepackage{dsfont}
\usepackage{bbm}
\usepackage{mathabx}
\allowdisplaybreaks

\renewcommand{\sinh}{\operatorname{sinh}}
\renewcommand{\cosh}{\operatorname{cosh}}

\newtheorem{thm}{Theorem}

\newtheorem{corollary}{Corollary}

\title{Killing random motions with finite velocity}

\author{ %\href{https://orcid.org/0000-0000-0000-0000}{Manfred Marvin Marchione}\thanks{Use footnote for providing further information about author (webpage, alternativeaddress)---\emph{not} for acknowledging funding agencies.}\\
	\href{https://orcid.org/0000-0002-6421-533X}{Enzo Orsingher} \\
	Department of Statistical Sciences\\
	Sapienza University of Rome\\
	\texttt{enzo.orsingher@uniroma1.it}
	\And
	\href{https://orcid.org/0000-0002-6163-044X}{Manfred Marvin Marchione}\\
	Independent Researcher\\
	\texttt{manfredmarvin.marchione@gmail.com}
}

\date{\today}

\renewcommand{\headeright}{}
\renewcommand{\undertitle}{}

\begin{document}
\maketitle

\begin{abstract}
In this paper, we study finite-velocity random motions subject to a killing mechanism acting at direction changes: at each change of direction, the motion is killed with a fixed probability and remains at its current position thereafter. We first consider the one-dimensional case of a standard telegraph process. We derive the exact distribution of the position at the killing time and show that it is Laplace distributed. We also characterize the law of the process conditional on survival up to a fixed time, proving that it coincides with that of a standard telegraph process driven by a thinned Poisson process with reduced switching intensity. We further derive the full time-dependent distribution of the killed process, including both its absolutely continuous and singular components. We then extend the analysis to a planar finite-velocity random motion with infinitely many possible directions. In this setting, we derive the distribution of the killed process and obtain an explicit representation of its density in terms of modified Bessel functions and a series involving Bessel polynomials.
\end{abstract}

%\keywords{Telegraph process, Dirichlet problem, Systems of differential equations}

\section{Introduction}
Finite-velocity random motions have been extensively studied in the probabilistic literature, giving rise to a broad range of models with different mechanisms governing their evolution. The classical univariate telegraph process can be traced back to the works by Goldstein \cite{goldstein} and Kac \cite{kac}, and its exact probability distribution was later derived by Orsingher \cite{orsingher}. The telegraph process with drift was introduced by Cane \cite{cane} and further developed Beghin et al. \cite{beghin2001probabilistic}.\\
Planar extensions have been proposed by several authors under different stochastic mechanisms. Planar random motions with constant speed and orthogonal directions were studied by Orsingher \cite{orsingher2000exact}, while motions with a minimal set of three directions were investigated by Di Crescenzo \cite{dicrminimal}. In the case with orthogonal directions, several modifications of the process have been proposed involving various mechanisms for the direction changes. A general framework was proposed by Orsingher and Marchione \cite{vortex}, who discussed a model in which clockwise and counterclockwise switches occur with different probabilities. Random motions with an arbitrary finite number of directions were proposed by Lachal \cite{lachal}, while Orsingher and De Gregorio \cite{orsingherdegreg} studied random flights with a continuum of possible directions, which are chosen with uniform distribution.\\
Other extensions of classical finite-velocity random motions include replacing the Poisson process governing direction changes with more general renewal processes, as in Martinucci et al. \cite{martinucci}, and allowing for space-varying velocities, as studied by Masoliver and Weiss \cite{masoliver}, Angelani et al. \cite{angelanietal} and Orsingher and Marchione \cite{OMSV}.\\
Additional lines of research aim at introducing random events that modify the evolution of the particle beyond its ordinary velocity changes. Ratanov \cite{ratanov} included a random jump component to the standard telegraph process, and discussed related applications to financial modeling. Stochastic resetting, where the particle is returned to a prescribed position at random times, was studied for one-dimensional run-and-tumble motions by Evans and Majumdar \cite{evans}. More recently, Di Crescenzo et al. \cite{iuliano} combined Poissonian resetting with a generalized telegraph motion driven by geometric counting processes. These studies illustrate the continuing interest in how changes to the switching mechanism and additional random interventions affect the distribution and properties of finite-velocity random motions.\\
In this paper, we investigate one-dimensional and planar random motions subject to a random killing mechanism. In particular, by assuming that the direction changes are paced by a Poisson process, we assume that every time a direction change occurs the process can be killed with a fixed probability $p\in(0,1)$, in the sense that with probability $p$ the process stops moving and remains at its current position thereafter.\\
\noindent We start by studying the univariate case. We consider a standard telegraph process $\{X(t)\}_{t\ge0}$ with initial position $X(0)=0$. At the initial time $t=0$ the process starts moving either rightwards or leftwards with probability $\frac{1}{2}$, and moves with constant velocity $c>0$. The process then switches direction according to a Poisson process $N(t)$ with constant intensity $\lambda$. Each time a direction change occurs, the telegraph process is killed with probability $p$. Denoting by $\mathcal{T}$ the random time at which the killing occurs, we represent the killed telegraph process as $X(t\wedge\mathcal{T})$.\\
Our first result is that the position of the telegraph process at the killing time follows a Laplace distribution with probability density function
\begin{equation}
\mathbb{P}\Big(X(\mathcal{T})\in dx\Big)/dx=\frac{\lambda\sqrt{p(2-p)}}{2c}\exp\left\{-\frac{\lambda\sqrt{p(2-p)}}{c}|x|\right\},\qquad x\in\mathbb{R}.\label{introtort}
\end{equation}
We also obtain the distribution of $X(\mathcal{T})$ conditional on the number of Poisson events up to the killing time and the initial direction of the motion. In the derivation of the probability density function (\ref{introtort}), we use some preliminary results concerning the Laplace distribution. In particular, given a sequence of independent Laplace random variables $L_k,\,k=1,2,...$, with common density
$$\mathbb{P}\Big(L_k\in dx\Big)=\frac{b}{2}\,e^{-b\,\lvert x\lvert}\;dx,\qquad x\in\mathbb{R}.$$
we use the fact that the sum $S_n=\sum_{k=1}^n L_k$ has probability density function
\begin{equation}\mathbb{P}\Big(S_n\in dx\Big)/dx=\left(\frac{b}{2}\right)^ne^{-b\lvert x\lvert}\sum_{k=0}^{n-1}{n-1+k \choose k} \frac{1}{(n-k-1)!}\left(\frac{1}{2b}\right)^k\lvert x\lvert^{n-k-1},\qquad x\in\mathbb{R}.\label{intro:sumlapl}\end{equation}
The derivation of the distribution (\ref{intro:sumlapl}) of the sum of independent Laplace random variables is a tricky problem, which has nevertheless already been solved in the literature. A standard reference is the monograph by Kotz et al. \cite{kotz}, where the result is obtained by using the properties of Bessel functions. Here we show that the same formula can be recovered in an elegant and direct way from the inversion of the Laplace transform, through an application of the residue theorem. This alternative derivation is presented in Appendix \ref{appendixLaplace} of the present paper.\\
We continue our analysis by investigating the exact distribution of the telegraph process $X(t)$ before the killing occurs. In particular, we show that, conditional on survival up to time $t$, the trajectory of the process on $[0,t]$ has the same law as that of a telegraph process with switching intensity $\lambda(1-p)$. In particular, the conditional distribution of $X(t)\lvert\{\mathcal{T}>t\}$ is supported on $[-ct,ct]$, with a singular component where $x=\pm ct.$ Denoting by $\delta(\cdot)$ the Dirac delta function, we prove that
\begin{align*}\mathbb{P}\Big(X(t)\in dx\;\Big\lvert\,\mathcal{T}>t\Big)/dx=&\frac{1}{2}e^{-\lambda(1-p)t}\big[\delta(x-ct)+\delta(x+ct)\big]\\
+\frac{e^{-\lambda(1-p)t}}{2c}&\mathbbm{1}_{\{|x|<ct\}}\Bigg[\lambda(1-p)I_0\left(\frac{\lambda(1-p)}{c}\sqrt{c^2t^2-x^2}\right)\\
&\qquad+\frac{\partial}{\partial t}I_0\left(\frac{\lambda(1-p)}{c}\sqrt{c^2t^2-x^2}\right)\Bigg].
\end{align*}
We provide an interpretation of this result in terms of thinning of Poisson processes. Combining the results presented above yields the distribution of the stopped process $X(t\wedge\mathcal{T})$, which reads
$$
\begin{aligned}
\frac{\mathbb{P}\Big(X(t\wedge\mathcal{T})\in dx\Big)}{dx}
={}&\frac{e^{-\lambda t}}{2}\big[\delta(x-ct)+\delta(x+ct)\big]\\
&+\frac{\mathbbm{1}_{\{|x|<ct\}}}{2c}\Bigg\{e^{-\lambda t}\left[\lambda I_0\left(\frac{\lambda(1-p)}{c}\sqrt{c^2t^2-x^2}\right)+\frac{\partial}{\partial t}I_0\left(\frac{\lambda(1-p)}{c}\sqrt{c^2t^2-x^2}\right)\right]\\
&\qquad+\lambda^2p(2-p)\int_{|x|/c}^{t}e^{-\lambda s}I_0\left(\frac{\lambda(1-p)}{c}\sqrt{c^2s^2-x^2}\right)\,ds\Bigg\}.
\end{aligned}
$$

In the remainder of the paper, we extend our analysis to a planar random motion with infinitely many possible directions. We consider a bivariate process $\{(X(t),Y(t))\}_{t\geq0}$ starting at the origin of the Cartesian plane and moving with constant speed $c>0$, with an initial direction chosen uniformly on $[0,2\pi)$. At each event of a Poisson process $N(t)$ with intensity $\lambda$, the motion is killed with probability $p$. Otherwise, a new direction is chosen uniformly and independently of the previous ones. As in the univariate case, we denote the killing time by $\mathcal{T}$.\\
Our main result for the planar model is an explicit representation of the probability density function of $(X(\mathcal{T}),Y(\mathcal{T}))$ in terms of modified Bessel functions and Bessel polynomials. In particular, we obtain, for $(x,y)\in\mathbb{R}^2$,
\begin{align*}
\frac{\mathbb{P}\Big(X(\mathcal{T})\in dx,\;Y(\mathcal{T})\in dy\Big)}{dx\,dy}=&\frac{\lambda^2p(1-p)}{2\pi c^2}K_0\left(\frac{\lambda}{c}\sqrt{p(2-p)}\,\sqrt{x^2+y^2}\right)\\
+\frac{\lambda p}{2\pi c}\frac{e^{-\frac{\lambda}{c}\,\sqrt{x^2+y^2}}}{\sqrt{x^2+y^2}}&\left[1+\frac{1}{2}\sum_{n=0}^{\infty}\left(\frac{2\lambda(1-p)^2\,\sqrt{x^2+y^2}}{c}\right)^{n+1}\frac{n!}{(2n+1)!}\,y_n\left(\frac{c}{\lambda \,\sqrt{x^2+y^2}}\right)\right]
\end{align*}
where $K_\nu$ denotes the modified Bessel function of the second kind and $y_n$ is the Bessel polynomial
$$y_n(z)=\sum_{k=0}^n\frac{(n+k)!}{(n-k)!\,k!}\left(\frac{z}{2}\right)^k,\qquad n\in\mathbb{N}_0.$$
The two terms in the density correspond respectively to killing at an even and an odd Poisson event. In particular, the contribution associated with an even number of events reduces to a single modified Bessel function, whereas the odd contribution admits a series representation in terms of Bessel polynomials.\\
To derive this result, we first study the distribution conditional on the number of Poisson events up to the killing time. For $n\geq1$, we obtain
\begin{align*}
&\frac{\mathbb{P}\Big(X(\mathcal{T})\in dx,\;Y(\mathcal{T})\in dy\;\Big\lvert\,N(\mathcal{T})=n\Big)}{dx\,dy}\\
&\qquad=\frac{1}{2^{n/2}\pi\,\Gamma(n/2)}\left(\frac{\lambda}{c}\right)^{n/2+1}(x^2+y^2)^{\frac{1}{2}\left(\frac{n}{2}-1\right)}K_{\frac{n}{2}-1}\left(\frac{\lambda}{c}\,\sqrt{x^2+y^2}\right).
\end{align*}
Averaging these conditional densities with respect to the geometric distribution of $N(\mathcal{T})$ yields the unconditional density presented above. We also recover the same expression by directly inverting the characteristic function
$$\mathbb{E}\left[e^{i[\alpha X(\mathcal{T})+\beta Y(\mathcal{T})]}\right]=\frac{\lambda p}{\sqrt{\lambda^2+c^2(\alpha^2+\beta^2)}-\lambda(1-p)},\qquad \alpha,\beta\in\mathbb{R}.$$
Both the conditional and unconditional distributions are rotationally invariant. Consequently, the distance from the origin and the angular component are independent, with the latter uniformly distributed on $[0,2\pi)$. We also derive the corresponding radial densities.\\
Finally, we investigate the distribution of the planar motion conditional on survival up to time $t$. The thinning interpretation developed in the univariate case applies here as well: conditional on $\mathcal{T}>t$, the trajectory on $[0,t]$ has the same law as that of the original planar motion with Poisson intensity reduced to $\lambda(1-p)$. Moreover, conditional on $N(t)=n$ and $\mathcal{T}>t$, the distribution is independent of both $\lambda$ and $p$. Thus, the conditional density obtained by Orsingher and De Gregorio \cite{orsingherdegreg} gives
$$\frac{\mathbb{P}\Big(X(t)\in dx,\;Y(t)\in dy\;\Big\lvert\,N(t)=n,\;\mathcal{T}>t\Big)}{dx\,dy}=\frac{n}{2\pi(ct)^n}\left(c^2t^2-x^2-y^2\right)^{n/2-1},\qquad x^2+y^2<c^2t^2,\quad n\geq1.$$
Averaging over the number of events of the thinned Poisson process and accounting for the singular component on the circumference $\sqrt{x^2+y^2}=ct$, corresponding to trajectories with no direction changes before time $t$, we obtain the explicit distribution conditional on survival:
\begin{align}\frac{\mathbb{P}\Big(X(t)\in dx,\;Y(t)\in dy\;\Big\lvert\,\mathcal{T}>t\Big)}{dx\,dy}&=\frac{e^{-\lambda(1-p)t}}{\pi}\,
\delta\left(x^2+y^2-c^2t^2\right)\nonumber\\
+\frac{\lambda(1-p)}{2\pi c}&\frac{\exp\left\{-\lambda(1-p)t+\frac{\lambda(1-p)}{c}\sqrt{c^2t^2-x^2-y^2}\right\}}{\sqrt{c^2t^2-x^2-y^2}},\qquad x^2+y^2<c^2t^2.\nonumber\end{align}
Based on the identity
$$
\begin{aligned}
&\mathbb{P}\big(X(t\wedge\mathcal{T})\in dx,Y(t\wedge\mathcal{T})\in dy\big)\\
&\qquad=
\mathbb{P}\big(X(t)\in dx,Y(t)\in dy,\mathcal{T}>t\big)
+
\int_0^t
\mathbb{P}\big(X(\mathcal{T})\in dx,Y(\mathcal{T})\in dy,\mathcal{T}\in ds\big).
\end{aligned}
$$
the results presented so far permit us to provide an explicit representation for the distribution of the killed planar motion $\mathbb{P}\big(X(t\wedge\mathcal{T})\in dx,Y(t\wedge\mathcal{T})\in dy\big)$.\\
The paper is organized in the following manner. In Section \ref{sec:uniT}, we introduce the problem in the univariate setting, define the killing mechanism, and derive the distribution of the process at the killing time. In Section \ref{sec:3}, we study the process conditional on survival and use the resulting distribution to obtain the law of the killed telegraph process. Section \ref{sec:4} extends these results to the planar setting.

\section{Distribution of the telegraph process at the killing time}\label{sec:uniT}
Consider a telegraph process $\{X(t)\}_{t\ge 0}$ with initial position $X(0)=0$. At time $t=0$, the process starts moving either to the right or to the left with probability $\frac{1}{2}$. It then moves with constant velocity $c>0$ and changes direction at random times governed by a homogeneous Poisson process. We denote by $N(t)$ the Poisson process that governs the direction changes, and by $\lambda>0$ its constant intensity. Throughout the paper, we assume that $N(t)$ is taken in its standard right-continuous version, with left limits. Moreover, the direction of motion at time $t$ is denoted by $D(t)$. We write $D(t)=d_0$ when $X(t)$ is moving to the right, and $D(t)=d_1$ when it is moving to the left.\\
In contrast to previous works in the literature, we assume that, whenever a Poisson event occurs, the telegraph process ${X(t)}_{t\ge 0}$ is killed with probability $p\in(0,1)$. In that case, the process does not move any further and remains permanently at its current position. We denote by $\mathcal{T}$ the time at which the telegraph process $X(t)$ is killed. By construction, the random variable $\mathcal{T}$ must coincide with one of the arrival times of the Poisson process $N(t)$. Formally, we define the arrival times as $$T_k=\inf\{t>0:\;N(t)=k\},\qquad k\in\mathbb{N}$$ and we adopt the convention that $T_0=0$. Observe that, in view of the assumption that the Poisson process is right-continuous, we have that $N(T_k)=k$ for all $k\in\mathbb{N}$. Conditionally on the arrival times $T_1,\,T_2,\,T_3,\,...\,$, it holds that
\begin{equation}\label{tauconditionaldef}\mathbb{P}\Big(\mathcal{T}=T_k\Big\lvert T_1,\,T_2,\,T_3,\,...\Big)=(1-p)^{k-1} p,\qquad k\in\mathbb{N}\end{equation}
An immediate consequence of formula (\ref{tauconditionaldef}) is that the number of Poisson events up to the killing time has a geometric distribution. Indeed,
\begin{equation}\label{geometricdistribution}\mathbb{P}\Big(N(\mathcal{T})=k\Big)=(1-p)^{k-1} p,\qquad k\in\mathbb{N}.\end{equation}
Moreover, it is straightforward to verify that the killing time $\mathcal{T}$ satisfies the identity
\begin{equation}\label{identityT}\mathcal{T}=\sum_{k=1}^{N(\mathcal{T})}(T_{k}-T_{k-1})\end{equation}
which will play a key role in our analysis.

We observe that formula (\ref{tauconditionaldef}) states that, conditionally on the Poisson arrival times, the killing time $\mathcal{T}$ has discrete distribution with support given by the set of the arrival times. However, its unconditional distribution is a continuous exponential distribution, as shown in the following result.
\begin{thm}The random variable $\mathcal{T}$ has exponential distribution with probability density function \begin{equation}\mathbb{P}\big(\mathcal{T}\in dt\big)/dt=\lambda p\,e^{-\lambda p t},\qquad t\in\mathbb{R}^+.\label{lpdens}\end{equation}\end{thm}
\begin{proof}We recall that, in view of the standard theory of Poisson processes, the interarrival times $T_k-T_{k-1}$, $k\in\mathbb{N}$, are independent and identically distributed exponential random variables with common characteristic function $$\mathbb{E}\Big[e^{i\alpha T_1}\Big]=\frac{\lambda}{\lambda-i\alpha},\qquad \alpha\in\mathbb{R}.$$
Thus, for $\alpha\in\mathbb{R}$, the identity (\ref{identityT}) permits us to express the characteristic function of $\mathcal{T}$ in the form
\begin{align}
\mathbb{E}\Big[e^{i\alpha \mathcal{T}}\Big]=&\mathbb{E}\Big[e^{i\alpha \sum_{k=1}^{N(\mathcal{T})}(T_{k}-T_{k-1})}\Big]=\mathbb{E}\left[\mathbb{E}\Big[e^{i\alpha T_1}\Big]^{N(\mathcal{T})}\right]\nonumber\\
=&\mathbb{E}\left[\left(\frac{\lambda}{\lambda-i\alpha}\right)^{N(\mathcal{T})}\right]=\frac{\lambda p}{\lambda p-i\alpha} \label{cflp}
\end{align}
where we have used the fact that $N(\mathcal{T})$ has a geometric distribution in the last step. The proof is completed by observing that the characteristic function (\ref{cflp}) is the Fourier transform of the probability density function (\ref{lpdens}).
\end{proof}
Having obtained the exact distribution of $\mathcal{T}$, we now investigate the distribution of the telegraph process at the killing time. In other words, we study the distribution of the random variable $X(\mathcal{T})$. To this aim, we first recall the following standard stochastic representation of the telegraph process:
\begin{equation}\label{telegraphstoch}X(t)=V(0)\int_0^t(-1)^{N(s)}ds\end{equation} where $V(0)$ is a binary random variable having distribution
\begin{equation}V(0)=\begin{cases}-c&\text{with prob. }\frac{1}{2}\\+c&\text{with prob. }\frac{1}{2}.\end{cases}\label{V0dis}\end{equation} Clearly,  the random variable $V(0)$ represents the initial velocity of the telegraph process, the sign of which depends on whether the process initially moves to the left or to the right. By combining formulas (\ref{identityT}) and (\ref{telegraphstoch}), the killed telegraph process can be expressed in the form
\begin{equation}X(\mathcal{T})=V(0)\sum_{k=1}^{N(\mathcal{T})}(-1)^{k+1}(T_{k}-T_{k-1}).\label{XTrepr}\end{equation}
Based on the representation (\ref{XTrepr}), we are able to study the distribution of $X(\mathcal{T})$ conditional on both the number of direction changes $N(\mathcal{T})$ at the killing time and the initial velocity $V(0)$ of the telegraph process. In particular, we consider, for $j=0,1$ and $n\in\mathbb{N}$, the probability density functions
$$f_{j,n}(x)=\mathbb{P}\Big(X(\mathcal{T})\in dx\,\Big\lvert D(0)=d_j,\;N(\mathcal{T})=n\Big)/dx,\qquad x\in\mathbb{R}.$$
By removing the conditioning with respect to the initial direction, we also define the density function
\begin{equation}f_n(x)=\mathbb{P}\Big(X(\mathcal{T})\in dx\,\Big\lvert N(\mathcal{T})=n\Big)/dx\label{fnuni}\end{equation}
In order to study the probability density functions defined above, we distinguish the two cases according to whether $N(\mathcal{T})$ is even or odd.\\
We start by analyzing the even case, i.e. the case in which $N(\mathcal{T})=2n$ for some $n\in\mathbb{N}$. Under this assumption, formula (\ref{XTrepr}) reduces to
\begin{align}X(\mathcal{T})\,\Big\lvert \{N(\mathcal{T})=2n\}\;=\,V(0)\sum_{k=1}^{n}\Big[(T_{2k-1}-T_{2k-2})-(T_{2k}-T_{2k-1})\Big].\label{evensumXT}\end{align}
We focus now our attention on the terms \begin{equation}L_k := (T_{2k-1}-T_{2k-2})-(T_{2k}-T_{2k-1}),\qquad k\in\mathbb{N}\label{lkdef}\end{equation} which appear in formula (\ref{evensumXT}). Since the Poisson interarrival times $T_j-T_{j-1},\;j\in\mathbb{N}$, are i.i.d. random variables with exponential distribution, it can be verified that the random variables $L_k$ follow a Laplace distribution with probability density function
$$\mathbb{P}\Big(L_k\in dx\Big)=\frac{\lambda}{2}\,e^{-\lambda\,\lvert x\lvert}\;dx,\qquad x\in\mathbb{R}.$$
Moreover, in view of the symmetry of the Laplace distribution, it follows immediately that the distribution of $V(0)\cdot L_k$ does not depend on the sign of $V(0)$ and is still a Laplace distribution with modified scale parameter. In particular, it holds that
\begin{equation}\mathbb{P}\Big(V(0)\cdot L_k\,\in dx\Big)=\mathbb{P}\Big(c\, L_k\,\in dx\Big)=\mathbb{P}\Big(-c\, L_k\,\in dx\Big)=\frac{\lambda}{2c}\,e^{-\frac{\lambda}{c}\,\lvert x\lvert}\;dx,\qquad x\in\mathbb{R}.\label{lxmor}\end{equation}
Thus, formula (\ref{evensumXT}) implies that, conditional on $N(\mathcal{T})=2n$, the distribution of $X(\mathcal{T})$ coincides with the distribution of the sum of $n$ independent Laplace random variables with probability density function (\ref{lxmor}). Hence, in view of formula (\ref{intro:sumlapl}), the probability density function (\ref{fnuni}) reads, for $n\in\mathbb{N}$,
\begin{equation}f_{2n}(x)=\left(\frac{\lambda}{2c}\right)^ne^{-\frac{\lambda}{c}\lvert x\lvert}\sum_{k=0}^{n-1}{n-1+k \choose k} \frac{1}{(n-k-1)!}\left(\frac{c}{2\lambda}\right)^k\lvert x\lvert^{n-k-1},\qquad x\in\mathbb{R}.\label{sumlaplacedensity}\end{equation}
Futhermore, it follows from formula (\ref{lxmor}) that $$f_{0,2n}(x)=f_{1,2n}(x)=f_{2n}(x)$$ which completes the analysis of the case in which $N(\mathcal{T})$ is an even number.\\
We now turn to the case in which the telegraph process is killed after an odd number of direction changes. Thus, we consider the distribution of $X(\mathcal{T})$ conditional of the event $N(\mathcal{T})=2n+1$, with $n\in\mathbb{N}_0$. In this case, formula (\ref{XTrepr}) can be represented as
\begin{align}X(\mathcal{T})\,\Big\lvert \{N(\mathcal{T})=2n+1\}\;=\,V(0)\,(T_{2n+1}-T_{2n})\,+\,V(0)\sum_{k=1}^{n}\Big[(T_{2k-1}-T_{2k-2})-(T_{2k}-T_{2k-1})\Big].\label{oddsumXT}\end{align}
The identity (\ref{oddsumXT}) implies that, if the distribution of $X(\mathcal{T})$ is conditioned on both an odd number of direction changes and the initial direction, its conditional distribution is no longer symmetric. The resulting asymmetric conditional distribution is derived in the following theorem.

\begin{thm}For $n\in\mathbb{N}$ and $x\in\mathbb{R}$, it holds that
\begin{align}
f_{0,2n+1}(x)=&\frac{\lambda e^{-\frac{\lambda}{c}\lvert x\lvert}}{4^n\,c}{2n-1\choose n-1}\nonumber\\
&+\mathbbm{1}_{\{x>0\}}\left(\frac{\lambda}{2c}\right)^n\frac{\lambda\,e^{-\frac{\lambda}{c}\lvert x\lvert}}{c\,(n-1)!}\sum_{k=0}^{n-1}\frac{(n-1+k)!}{(n-k)!\,k!}\left(\frac{c}{2\lambda}\right)^k \lvert x\lvert^{n-k}\nonumber\\
&+\mathbbm{1}_{\{x<0\}}\left(\frac{\lambda}{2c}\right)^n\frac{\lambda\,e^{-\frac{\lambda}{c}\lvert x\lvert}}{c\,n!}\sum_{r=0}^{n-2}\frac{(n+r)!}{(n-r-1)!\,r!}\left(\frac{c}{2\lambda}\right)^{r+1}\lvert x\lvert^{n-r-1}.\label{delfmigun}
\end{align}

Moreover, \begin{equation}f_{1,2n+1}(x)=f_{0,2n+1}(-x).\label{delfmigdu}\end{equation}
\end{thm}
\begin{proof}By defining the random variable \begin{equation}E_{2n+1}:=c\,(T_{2n+1}-T_{2n})\label{E2n+1def}\end{equation} and taking into account that the Poisson interarrival times are exponentially distributed, the probability density function of $E_{2n+1}$ is \begin{equation}h(x)=\frac{\lambda}{c}\,e^{-\frac{\lambda}{c} x},\qquad x\in\mathbb{R}^+.\label{taradatta}\end{equation} By conditioning equation (\ref{oddsumXT}) on the initial direction being $d_0$, we have that
$$X(\mathcal{T})\,\Big\lvert \{N(\mathcal{T})=2n+1,\,D(0)=d_0\}\;=\,E_{2n+1}+c\sum_{k=1}^{n}L_k$$ where the random variables $L_k,\,k=1,...,n$, are defined in formula (\ref{lkdef}). In other words, the conditional random variable $X(\mathcal{T})\,\Big\lvert \{N(\mathcal{T})=2n+1,\,D(0)=d_0\}$ coincides with the sum of $n$ independent Laplace random variables with common density function (\ref{lxmor}) and an exponentally distributed random variable with density (\ref{taradatta}). Thus, by taking into account that the sum of the Laplace variables has probability density function (\ref{sumlaplacedensity}), the conditional density of $X(\mathcal{T})$ can be expressed by means of the convolution integral
\begin{align}f_{0,2n+1}(x)&=\int_0^{+\infty}f_{2n}(x-z)\,h(z)\,dz\nonumber\\
&=\left(\frac{\lambda}{2c}\right)^n\frac{\lambda}{c\,(n-1)!}\,\sum_{k=0}^{n-1}\frac{(n-1+k)!}{(n-1-k)!\;k!}\;\left(\frac{c}{2\lambda}\right)^k\int_0^{+\infty}e^{-\frac{\lambda}{c} \lvert x-z \lvert-\frac{\lambda}{c} z}\,\lvert x-z\lvert^{n-k-1}\mathop{dz}.\label{intsts}\end{align}
To reformulate the expression (\ref{intsts}), we consider the cases of positive and negative values of $x$ separately.\\
\noindent For $x>0$, we can write that 
\begin{align}
f_{0,2n+1}(x)&=\left(\frac{\lambda}{2c}\right)^n\frac{\lambda\,e^{-\frac{\lambda}{c}x}}{c\,(n-1)!}\,\sum_{k=0}^{n-1}\frac{(n-1+k)!}{(n-1-k)!\;k!}\;\left(\frac{c}{2\lambda}\right)^k\int_0^{x}(x-z)^{n-k-1}\mathop{dz}\nonumber\\
&\qquad\qquad\qquad+\left(\frac{\lambda}{2c}\right)^n\frac{\lambda\,e^{-\frac{\lambda}{c}x}}{c\,(n-1)!}\,\sum_{k=0}^{n-1}\frac{(n-1+k)!}{(n-1-k)!\;k!}\;\left(\frac{c}{2\lambda}\right)^k\int_0^{+\infty}e^{-\frac{2\lambda}{c}z}\,z^{n-k-1}\mathop{dz}\nonumber\\
&=\left(\frac{\lambda}{2c}\right)^n\frac{\lambda\,e^{-\frac{\lambda}{c}x}}{c\,(n-1)!}\,\sum_{k=0}^{n-1}\frac{(n-1+k)!}{(n-k)!\;k!}\;\left(\frac{c}{2\lambda}\right)^k x^{n-k}+\frac{1}{4^n}\frac{\lambda\,e^{-\frac{\lambda}{c}x}}{c\,(n-1)!}\,\sum_{k=0}^{n-1}\frac{(n-1+k)!}{k!}\nonumber\\
&=\left(\frac{\lambda}{2c}\right)^n\frac{\lambda\,e^{-\frac{\lambda}{c}x}}{c\,(n-1)!}\,\sum_{k=0}^{n-1}\frac{(n-1+k)!}{(n-k)!\;k!}\;\left(\frac{c}{2\lambda}\right)^k x^{n-k}+\frac{\lambda\,e^{-\frac{\lambda}{c}x}}{c\,4^n}\,{2n-1\choose n-1},\qquad x>0\label{smokdisp}
\end{align}
where we have used the hockey-stick identity
\begin{equation}\sum_{j=0}^r{m+j \choose j}={m+r+1 \choose r}.\label{hockstiid}\end{equation}
Similarly, in the case $x<0$, it holds that 
\begin{align}
f_{0,2n+1}(x)&=\left(\frac{\lambda}{2c}\right)^n\frac{\lambda\,e^{\frac{\lambda}{c}x}}{c\,(n-1)!}\,\sum_{k=0}^{n-1}\frac{(n-1+k)!}{(n-1-k)!\;k!}\;\left(\frac{c}{2\lambda}\right)^k\int_0^{+\infty}e^{-\frac{2\lambda}{c}z}\,(z-x)^{n-k-1}\mathop{dz}\nonumber\\
&=\left(\frac{\lambda}{2c}\right)^n\frac{\lambda\,e^{\frac{\lambda}{c}x}}{c\,(n-1)!}\,\sum_{k=0}^{n-1}\frac{(n-1+k)!}{(n-1-k)!\;k!}\;\left(\frac{c}{2\lambda}\right)^k\sum_{j=0}^{n-k-1}{n-k-1 \choose j}\;j!\;\left(\frac{c}{2\lambda}\right)^{j+1}(-x)^{n-k-j-1}\nonumber\\
&=\left(\frac{\lambda}{2c}\right)^n\frac{\lambda\,e^{\frac{\lambda}{c}x}}{c}\,\sum_{r=0}^{n-1}\frac{(-x)^{n-r-1}}{(n-r-1)!}\;\left(\frac{c}{2\lambda}\right)^{r+1}\sum_{j=0}^{r}{n-1+r-j \choose r-j}\nonumber\\
&=\left(\frac{\lambda}{2c}\right)^n\frac{\lambda\,e^{\frac{\lambda}{c}x}}{c\,n!}\,\sum_{r=0}^{n-1}\frac{(n+r)!}{(n-r-1)!\;r!}\;\left(\frac{c}{2\lambda}\right)^{r+1}(-x)^{n-r-1},\qquad x<0\label{smokdisplan}
\end{align} where we used the identity (\ref{hockstiid}) in the last step. The result (\ref{delfmigun}) now follows by combining formulas (\ref{smokdisp}) and (\ref{smokdisplan}) in a suitable manner. Finally, the identity (\ref{delfmigdu}) is a direct consequence of the fact that $V(0)$ has distribution (\ref{V0dis}).
\end{proof}
Having obtained the conditional distribution of the killed telegraph process $X(\mathcal{T})$ given its initial direction when the number of direction changes is odd, we now remove the conditioning on the initial direction. Interestingly, when $N(\mathcal{T})=2n+1$, $n\in\mathbb{N}_0$, the resulting distribution again coincides with that of the sum of $n+1$ independent Laplace random variables, as shown in the following result.
\begin{corollary}\label{cor:olap}For $n\in\mathbb{N}_0$ and $x\in\mathbb{R}$, it holds that
$$f_{2n+1}(x)=\left(\frac{\lambda}{2c}\right)^{n+1}e^{-\frac{\lambda}{c}\lvert x \lvert}\sum_{k=0}^{n}{n+k\choose k}\,\frac{1}{(n-k)!}\,\left(\frac{c}{2\lambda}\right)^{k}\lvert x\lvert ^{n-k}.$$
\end{corollary}
\begin{proof}
For $n=0$, the result is trivial. For $n\ge1$, it holds that
\begin{align}
f_{2n+1}(x)=&\frac{1}{2}\big[f_{0,2n+1}(x)+f_{1,2n+1}(x)\big]=\frac{1}{2}\big[f_{0,2n+1}(x)+f_{0,2n+1}(-x)\big]\nonumber\\
=&\frac{\lambda e^{-\frac{\lambda}{c}\lvert x\lvert}}{4^n\,c}{2n-1\choose n-1}+\left(\frac{\lambda}{2c}\right)^{n+1}\frac{e^{-\frac{\lambda}{c}\lvert x\lvert}}{(n-1)!}\sum_{k=0}^{n-1}\frac{(n-1+k)!}{(n-k)!\,k!}\left(\frac{c}{2\lambda}\right)^k \lvert x\lvert^{n-k}\nonumber\\
&\qquad\qquad+\left(\frac{\lambda}{2c}\right)^{n+1}\frac{e^{-\frac{\lambda}{c}\lvert x\lvert}}{n!}\sum_{r=0}^{n-2}\frac{(n+r)!}{(n-r-1)!\,r!}\left(\frac{c}{2\lambda}\right)^{r+1}\lvert x\lvert^{n-r-1}\nonumber\\
=&\frac{\lambda e^{-\frac{\lambda}{c}\lvert x\lvert}}{4^n\,c}{2n-1\choose n-1}+\left(\frac{\lambda}{2c}\right)^{n+1}\frac{e^{-\frac{\lambda}{c}\lvert x\lvert}}{n!} \;\lvert x\lvert^{n}\nonumber\\
&\qquad\qquad+\left(\frac{\lambda}{2c}\right)^{n+1}\frac{e^{-\frac{\lambda}{c}\lvert x\lvert}}{(n-1)!}\sum_{k=1}^{n-1}\frac{(n-1+k)!}{(n-k)!\,k!}\left(\frac{c}{2\lambda}\right)^k \lvert x\lvert^{n-k}\nonumber\\
&\qquad\qquad+\left(\frac{\lambda}{2c}\right)^{n+1}\frac{e^{-\frac{\lambda}{c}\lvert x\lvert}}{n!}\sum_{r=0}^{n-2}\frac{(n+r)!}{(n-r-1)!\,r!}\left(\frac{c}{2\lambda}\right)^{r+1}\lvert x\lvert^{n-r-1}\nonumber\\
=&\frac{\lambda e^{-\frac{\lambda}{c}\lvert x\lvert}}{4^n\,c}{2n-1\choose n-1}+\left(\frac{\lambda}{2c}\right)^{n+1}\frac{e^{-\frac{\lambda}{c}\lvert x\lvert}}{n!} \;\lvert x\lvert^{n}\nonumber\\
&\qquad\qquad+\left(\frac{\lambda}{2c}\right)^{n+1}\,e^{-\frac{\lambda}{c}\lvert x\lvert}\,\sum_{k=1}^{n-1}{n+k\choose n}\frac{1}{(n-k)!}\left(\frac{c}{2\lambda}\right)^k \lvert x\lvert^{n-k}\nonumber\\
=&\left(\frac{\lambda}{2c}\right)^{n+1}\,e^{-\frac{\lambda}{c}\lvert x\lvert}\,\sum_{k=0}^{n}{n+k\choose n}\frac{1}{(n-k)!}\left(\frac{c}{2\lambda}\right)^k \lvert x\lvert^{n-k}\nonumber
\end{align}
where we have used the identity $${2n-1\choose n-1}=\frac{1}{2}{2n\choose n}.$$ This completes the proof.
\end{proof}
We now make some remarks on Corollary \ref{cor:olap}. To derive the distribution of $X(\mathcal{T})$ conditional on $N(\mathcal{T})=2n+1$, we first conditioned on the initial velocity $V(0)$ of the telegraph process and then averaged over its distribution. This approach permits us to obtain the explicit asymmetric distribution arising when the initial velocity is given. We emphasize, however, that the distribution of $X(\mathcal{T})\mid\{N(\mathcal{T})=2n+1\},$
with the initial velocity left unspecified, can also be derived directly, without first conditioning on $V(0)$. In order to do this, we observe that formula (\ref{oddsumXT}) can be expressed in the form
\begin{equation}X(\mathcal{T})\mid\{N(\mathcal{T})=2n+1\}\;=\,\frac{V(0)}{c}\, E_{2n+1}+V(0)\sum_{k=1}^{n}L_k\label{LoiLan}\end{equation}
where $L_k,\,k\in\mathbb{N},$ and  $E_{2n+1}$ are defined in formulas (\ref{lkdef}) and (\ref{E2n+1def}) respectively. Observe that the variables $V(0)\,L_k,\,k\in\mathbb{N},$ and $\frac{V(0)}{c}\,E_{2n+1}$ all follow a Laplace distribution with probability density function (\ref{lxmor}). Moreover, by using standard properties of the Poisson process, it can be verified that these variables are mutually independent. Thus, the identity (\ref{LoiLan}) implies that $X(\mathcal{T})\mid\{N(\mathcal{T})=2n+1\}$ is identical in distribution to the sum of $n+1$ independent Laplace random variables. Therefore, the result stated in Corollary \ref{cor:olap} can be derived directly by applying formula (\ref{intro:sumlapl}).\\
The next step in our analysis is to determine the distribution of the killed telegraph process $X(\mathcal{T})$ conditional only on the parity of the number of direction changes, that is, on the events that $N(\mathcal{T})$ is even or odd. In particular, we consider the probability density functions
$$f_{\mathrm{even}}(x)=\mathbb{P}\Big(X(\mathcal{T})\in dx,\;N(\mathcal{T})\ \text{is even}\Big)/dx$$
and
$$f_{\mathrm{odd}}(x)=\mathbb{P}\Big(X(\mathcal{T})\in dx,\;N(\mathcal{T})\ \text{is odd}\Big)/dx.$$
The explicit form of these density functions is given in the following theorem.
\begin{thm}For $x\in\mathbb{R}$, it holds that
$$f_{\mathrm{even}}(x)=\frac{\lambda(1-p)}{2c}\sqrt{\frac{p}{2-p}}\exp\left\{-\frac{\lambda}{c}\sqrt{p(2-p)}\,\lvert x\lvert\right\}$$
and 
$$f_{\mathrm{odd}}(x)=\frac{\lambda}{2c}\sqrt{\frac{p}{2-p}}\exp\left\{-\frac{\lambda}{c}\sqrt{p(2-p)}\,\lvert x\lvert\right\}.$$
\end{thm}
\begin{proof}By recalling that the number of direction changes $N(\mathcal{T})$ before the killing has geometric distribution (\ref{geometricdistribution}), and employing formula \ref{sumlaplacedensity}, we can write that
\begin{align}
f_{\mathrm{even}}(x)&=\sum_{n=1}^{\infty}p(1-p)^{2n-1}f_{2n}(x)\nonumber\\
&=p\,e^{-\frac{\lambda}{c}\lvert x\lvert}\sum_{n=1}^{\infty}\frac{(1-p)^{2n-1}}{(n-1)!}\sum_{k=0}^{n-1}\frac{1}{4^k}\left(\frac{\lambda}{2c}\right)^{n-k}\frac{(n-1+k)!}{(n-1-k)!\,k!}\lvert x\lvert^{n-k-1}\nonumber\\
&=p\,e^{-\frac{\lambda}{c}\lvert x\lvert}\sum_{k=0}^{\infty}\sum_{n=k+1}^{\infty}\frac{(1-p)^{2n-1}}{4^k}\left(\frac{\lambda}{2c}\right)^{n-k}\frac{(n-1+k)!}{(n-1-k)!\,k!\,(n-1)!}\lvert x\lvert^{n-k-1}\nonumber\\
&=p\,e^{-\frac{\lambda}{c}\lvert x\lvert}\sum_{k=0}^{\infty}\sum_{n=0}^{\infty}\frac{(1-p)^{2(n+k)+1}}{4^k}\left(\frac{\lambda}{2c}\right)^{n+1}\frac{(n+2k)!}{n!\,k!\,(n+k)!}\lvert x\lvert^n\nonumber\\
&=\frac{\lambda p(1-p)}{2c}\,e^{-\frac{\lambda}{c}\lvert x\lvert}\sum_{n=0}^{\infty}\frac{1}{n!}\left(\frac{\lambda(1-p)^2}{2c}\lvert x\lvert\right)^n\sum_{k=0}^{\infty}\binom{n+2k}{k}\left(\frac{(1-p)^2}{4}\right)^k.\label{stepleb}\end{align}
By using the identity (see Lebedev \cite{lebedev}, p. 253)
$$\sum_{k=0}^{\infty}{n+2k\choose k}t^k=\frac{1}{\sqrt{1-4t}}\left(\frac{2}{1+\sqrt{1-4t}}\right)^n\qquad\lvert t\lvert<\frac{1}{4}$$
formula (\ref{stepleb}) reduces to
\begin{align}f_{\mathrm{even}}(x)&=\frac{\lambda p(1-p)}{2c}\,e^{-\frac{\lambda}{c}\lvert x\lvert}\sum_{n=0}^{\infty}\frac{1}{n!}\left(\frac{\lambda(1-p)^2}{2c}\lvert x\lvert\right)^n\frac{1}{\sqrt{1-(1-p)^2}}\left(\frac{2}{1+\sqrt{1-(1-p)^2}}\right)^n\nonumber\\
&=\frac{\lambda(1-p)}{2c}\sqrt{\frac{p}{2-p}}\,e^{-\frac{\lambda}{c}\lvert x\lvert}\sum_{n=0}^{\infty}\frac{1}{n!}\left(\frac{\lambda}{c}\left(1-\sqrt{p(2-p)}\right)\lvert x\lvert\right)^n\nonumber\\
&=\frac{\lambda(1-p)}{2c}\sqrt{\frac{p}{2-p}}\exp\left\{-\frac{\lambda}{c}\sqrt{p(2-p)}\,\lvert x\lvert\right\}.\nonumber
\end{align}
Similarly, when the number of direction changes is odd, it holds that
\begin{align}
f_{\mathrm{odd}}(x)&=\sum_{n=0}^{\infty}p(1-p)^{2n}f_{2n+1}(x)\nonumber\\
&=\sum_{n=0}^{\infty}p(1-p)^{2n}\left(\frac{\lambda}{2c}\right)^{n+1}e^{-\frac{\lambda}{c}\lvert x\lvert}\sum_{k=0}^{n}\binom{n+k}{k}\left(\frac{c}{2\lambda}\right)^k\frac{\lvert x\lvert^{n-k}}{(n-k)!}\\
&=\frac{\lambda p}{2c}\,e^{-\frac{\lambda}{c}\lvert x\lvert}\sum_{k=0}^{\infty}\sum_{n=k}^{\infty}\frac{(1-p)^{2n}}{4^k}\left(\frac{\lambda}{2c}\right)^{n-k}\binom{n+k}{k}\frac{\lvert x\lvert^{n-k}}{(n-k)!}\nonumber\\
&=\frac{\lambda p}{2c}\,e^{-\frac{\lambda}{c}\lvert x\lvert}\sum_{k=0}^{\infty}\sum_{n=0}^{\infty}\frac{(1-p)^{2(n+k)}}{4^k}\left(\frac{\lambda}{2c}\right)^n\binom{n+2k}{k}\frac{\lvert x\lvert^n}{n!}\nonumber\\
&=\frac{\lambda p}{2c}\,e^{-\frac{\lambda}{c}\lvert x\lvert}\sum_{n=0}^{\infty}\frac{1}{n!}\left(\frac{\lambda(1-p)^2}{2c}\lvert x\lvert\right)^n\sum_{k=0}^{\infty}\binom{n+2k}{k}\left(\frac{(1-p)^2}{4}\right)^k\nonumber\\
&=\frac{\lambda p}{2c}\,e^{-\frac{\lambda}{c}\lvert x\lvert}\sum_{n=0}^{\infty}\frac{1}{n!}\left(\frac{\lambda(1-p)^2}{2c}\lvert x\lvert\right)^n\frac{1}{\sqrt{1-(1-p)^2}}\left(\frac{2}{1+\sqrt{1-(1-p)^2}}\right)^n\nonumber\\
&=\frac{\lambda}{2c}\sqrt{\frac{p}{2-p}}\,e^{-\frac{\lambda}{c}\lvert x\lvert}\sum_{n=0}^{\infty}\frac{1}{n!}\left(\frac{\lambda}{c}\left(1-\sqrt{p(2-p)}\right)\lvert x\lvert\right)^n\nonumber\\
&=\frac{\lambda}{2c}\sqrt{\frac{p}{2-p}}\exp\left\{-\frac{\lambda}{c}\sqrt{p(2-p)}\,\lvert x\lvert\right\}\nonumber
\end{align} which completes the proof.
\end{proof}
We are finally able to obtain the unconditional distribution of the killed telegraph process.
\begin{thm}\label{finallucy}The random variable $X(\mathcal{T})$ has probability density function
\begin{equation}f(x)=\frac{\lambda}{2c}\sqrt{p(2-p)}\exp\left\{-\frac{\lambda}{c}\sqrt{p(2-p)}\,\lvert x\lvert\right\},\qquad x\in\mathbb{R}.\label{ffubit}\end{equation}
\end{thm}
\begin{proof}The result follows immediately by observing that $$f(x)=f_{\mathrm{even}}(x)+f_{\mathrm{odd}}(x).$$
\end{proof}
In view of Theorem \ref{finallucy}, we have obtained the exact distribution of the killed telegraph process $X(\mathcal{T})$, which turns out to be a Laplace distribution. Before concluding this section, we investigate the asymptotic behaviour of this distribution in the hydrodynamic limit. In the literature on finite-velocity random motions, it is well known that, when both the velocity $c$ and the intensity of the direction changes $\lambda$ tend to infinity under the scaling $\frac{\lambda}{c^2}\to1$, the telegraph process $\{X(t)\}_{t\ge0}$ converges in distribution to a standard Brownian motion. This fact we first noted by Kac \cite{kac}, who showed that the hyperbolic partial differential equation gonverning the distribution of the telegraph process reduces to the heat equation in the hydrodynamic limit. Orsingher \cite{orsingher} later proved that the probability density function of the telegraph process admits an explicit expression in terms of modified Bessel functions of the first kind, and was also able to calculate the hydrodynamic limit of such a density. In particular, the author proved that, for all $t>0$,
$$\lim_{\substack{\lambda,c\to+\infty\\\lambda/c^2\to1}} X(t) \overset{i.d.}{=}B(t)$$ where $\{B(t)\}_{t\ge0}$ is a standard Brownian motion. Our aim is to extend this asymptotic analysis to the killed telegraph process $X(\mathcal{T})$. However, this analysis requires some preliminary remarks. We recall that, within the framework presented in the present paper, the killing of the telegraph process may occur with probability $p$ every time a direction change occurs. Since the analysis of hydrodynamic limit involves taking the limit for $\lambda\to+\infty$, the direction changes become infinitely frequent in the asymptotic regime, leading to the telegraph process being killed immediately after the initial time 0. In order to verify this, it is sufficient to verify that the killing time $\mathcal{T}$ converges in probability to 0 as $\lambda\to+\infty$. Indeed, by recalling that $\mathcal{T}$ follows the exponential distribution (\ref{lpdens}), it is straightforward to show that $$\lim_{\lambda\to+\infty}\mathbb{P}\Big(\mathcal{T}<\varepsilon\Big)=1,\qquad\forall\varepsilon>0.$$ In order to avoid the telegraph process being killed immediately, one must assume that, as $\lambda\to+\infty$, the probability of $X(t)$ being killed at each direction change converges to 0 with a suitable convergence rate, which permits obtaining a non-degenerate limiting distribution. It turns out that a nontrivial limiting distribution is achieved by considering the limit for $\lambda,c\to+\infty$ and $p\to0$ in such a way that $\frac{\lambda}{c^2}\to1$ and $\lambda p\to1$. Under this scaling, the killing time $\mathcal{T}$ preserves the exponential distribution in the hydrodynamic limit, with asymptotic probability density function
\begin{equation}\lim_{\substack{\lambda\to+\infty\\p\to0}}\mathbb{P}\Big(\mathcal{T}\in dt\Big)/dt=e^{-t},\qquad t\in\mathbb{R}^+.\label{lalaldre}\end{equation}
Moreover, in view of formula (\ref{ffubit}), it is straightforward to verify that the killed telegraph process $X(\mathcal{T})$ has asymptotic density function
\begin{equation}\label{hydrofinalsecun}\lim_{\substack{\lambda,c\to+\infty\\p\to0}}f(x)=\frac{1}{\sqrt{2}}\exp\left\{-\sqrt{2}\,\lvert x\lvert\right\},\qquad x\in\mathbb{R}\end{equation} which is, of course, the probability density function of a Laplace distribution. We also emphasize that formula (\ref{hydrofinalsecun}) provides an interesting information about the asymptotic dependence between the telegraph process $X(t)$ and its killing time $\mathcal{T}$ in the hydrodynamic limit. Indeed, since the telegraph process converges to a Brownian motion and $\mathcal{T}$ to an exponential distribution, the limiting density function (\ref{hydrofinalsecun}) describes the distribution of a Brownian motion at a random exponential time. More precisely, we remark that the probability density function (\ref{hydrofinalsecun}) is obtained by assuming that the Brownian motion and the exponential time are independent. In fact, given a Brownian motion $\{B(t)\}_{t\ge0}$ and an exponential random variable $E$ independent from $B(t)$ and having probability density function (\ref{lalaldre}), it can be verified that $B(E)$ has probability density function (\ref{hydrofinalsecun}). Thus, formula (\ref{hydrofinalsecun}) suggests that the telegraph process and its killing time become independent in the hydrodynamic limit.

\section{Distribution of the telegraph process before the killing time}\label{sec:3}
In the analysis performed so far, we have assumed that the telegraph process $X(t)$ may be killed with probaility $p$ each time it changes its direction, and we have obtained the exact distribution of its position at the time the killing occurs. In this section, we are interested in studying the exact distribution of the telegraph process before it is killed. We first observe that, for a fixed time $t$ prior to the killing of the process, the support of $X(t)$ coincides with the interval $[-ct,ct]$. This follows immediately by the fact that the telegraph process has constant velocity $c$. Moreover, the distribution of $X(t)$ has a singular components on the endpoints of the support. Indeed, the position of the process at time $t$ coincides with one of the endpoints if and only if no changes of direction have occurred up to time $t$. In order to study the distribution of $X(t)$ before the killing time $\mathcal{T}$, we first consider the absolutely continuous component of the distribution in the interior of the support. Thus, we consider the probability density function
\begin{equation}g(x,t)=\mathbb{P}\Big(X(t)\in\mathop{dx}\big\lvert\,\mathcal{T}>t\Big)/dx,\qquad x\in(-ct,ct).\label{gdefiunivar}\end{equation}
In order to determine the density (\ref{gdefiunivar}), we start by considering the joint distribution of the telegraph process and its direction, that is
\begin{equation}\label{gjdedur}g_j(x,t)=\mathbb{P}\Big(X(t)\in\mathop{dx},\;D(t)=d_j\,\big\lvert\,\mathcal{T}>t\Big)/dx,\qquad x\in(-ct,ct).\end{equation}
We claim that the densities (\ref{gjdedur}) satisfy the linear system
\begin{equation}\begin{cases}
g_0(x,t+dt)=g_0(x-c\,dt,t)\;(1-\lambda(1-p)\,dt)+g_1(x,t)\,\lambda\,(1-p)\,dt+o(dt)\\[0.2em]
g_1(x,t+dt)=g_1(x+c\,dt,t)\;(1-\lambda(1-p)\,dt)+g_0(x,t)\,\lambda\,(1-p)\,dt+o(dt).
\end{cases}\label{judover}\end{equation}
We only present the argument which permits obtaining the first equation of the system (\ref{judover}), as the second equation follows in a similar manner. Observe that
\begin{align}
g_0(x,t+&dt)\,dx=\mathbb{P}\Big(X(t+dt)\in\mathop{dx},\;D(t+dt)=d_j,\;\mathcal{T}>t+dt\Big)/\mathbb{P}\Big(\mathcal{T}>t+dt\Big)\nonumber\\
=&\mathbb{P}\Big(X(t+dt)\in\mathop{dx},\;D(t+dt)=d_0,\;N(t+dt)-N(t)=0,\;\mathcal{T}>t+dt\Big)\,e^{\lambda p(t+dt)}\nonumber\\
&\quad+\mathbb{P}\Big(X(t+dt)\in\mathop{dx},\;D(t+dt)=d_0,\;N(t+dt)-N(t)=1,\;\mathcal{T}>t+dt\Big)\,e^{\lambda p(t+dt)}+o(dt)\nonumber\\
=&\mathbb{P}\Big(X(t)+c\,dt\in\mathop{dx},\;D(t)=d_0,\;N(t+dt)-N(t)=0,\;\mathcal{T}>t\Big)\,e^{\lambda p(t+dt)}\nonumber\\
&\quad+\mathbb{P}\Big(X(t)\in\mathop{dx},\;D(t)=d_1,\;N(t+dt)-N(t)=1,\;\mathcal{T}>t,\,\mathcal{T}\notin[t,t+dt)\Big)\,e^{\lambda p(t+dt)}+o(dt)\nonumber\\
=&\mathbb{P}\Big(X(t)+c\,dt\in\mathop{dx},\;D(t)=d_0\,\big\lvert\,\mathcal{T}>t\Big)\cdot\mathbb{P}\Big(\;N(t+dt)-N(t)=0\,\big\lvert\,\mathcal{T}>t\Big)\,e^{\lambda p\,dt}\nonumber\\
&\quad+\mathbb{P}\Big(X(t)\in\mathop{dx},\;D(t)=d_1\,\big\lvert\,\mathcal{T}>t\Big)\cdot\mathbb{P}\Big(N(t+dt)-N(t)=1,\;\mathcal{T}\notin[t,t+dt)\,\Big\lvert\,\mathcal{T}>t\Big)\,e^{\lambda p\,dt}+o(dt)\nonumber\\
=&g_0(x-c\,dt,t)\;(1-\lambda\,dt)(1+\lambda p\,dt)+g_1(x,t)\,\lambda\,(1-p)\,dt+o(dt)\nonumber
\end{align}
which completes the proof of the first equation in (\ref{judover}). The second equation can be obtained by means of similar arguments. The next step of our analysis is to take a first-order Taylor expansion of (\ref{judover}), which yields the linear system of partial differential equations
\begin{equation}\begin{dcases}
\frac{\partial g_0}{\partial t}=-c\,\frac{\partial g_0}{\partial x}+\lambda\,(1-p)\,(g_1-g_0)\\[0.2em]
\frac{\partial g_1}{\partial t}=c\,\frac{\partial g_1}{\partial x}+\lambda\,(1-p)\,(g_0-g_1)
\end{dcases}\label{judoverlim}\end{equation}
By combining the equations of the system (\ref{judoverlim}), it can be verified that the probability density function (\ref{gdefiunivar}) satisfies the partial differential equation
\begin{equation}\frac{\partial^2 g}{\partial t^2}+2\lambda(1-p)\frac{\partial g}{\partial t}=c^2\frac{\partial^2g}{\partial x^2}.\label{gpde}\end{equation}
We are now able to make some crucial remarks. Both the system (\ref{judoverlim}) and equation (\ref{gpde}) coincide with the equations which govern the distribution of a telegraph process whose direction changes are paced by a Poisson process with intensity $\lambda(1-p)$. This fact shows that, if we consider a standard telegraph process $X(t)$ with switching intensity $\lambda$ subject it to the killing mechanism introduced here, and condition on killing not having occurred by the current time $t$, the resulting random motion is still a standard telegraph process, but with modified switching intensity $\lambda(1-p)$. Further confirmation of this result can be obtained by examining the singular component of the distribution. As mentioned above, conditional on killing not having occurred by time $t$, the telegraph process $X(t)$ has support $[-ct,ct]$ and possesses probability masses at the endpoints of this interval. To recover this singular component, it is sufficient to observe that
\begin{align}
\mathbb{P}\Big(X(t)=ct\,\Big\lvert\,\mathcal{T}>t\Big)&=\mathbb{P}\Big(N(t)=0,\,D(0)=d_0\,\Big\lvert\,\mathcal{T}>t\Big)=\frac{\mathbb{P}\Big(N(t)=0,\,D(0)=d_0,\,\mathcal{T}>t\Big)}{\mathbb{P}\big(\mathcal{T}>t\big)}\nonumber\\
&=\frac{\mathbb{P}\Big(N(t)=0,\mathcal{T}>t\Big)}{2\,e^{-\lambda pt}}
=\frac{\mathbb{P}\big(N(t)=0\big)}{2\,e^{-\lambda pt}}
=\frac{1}{2}\,e^{-\lambda(1-p)t}.
\end{align}
An analogous argument yields the probability mass at the lower endpoint. Hence,
\begin{equation}\mathbb{P}\Big(X(t)=ct\;\Big\lvert\,\mathcal{T}>t\Big)=\mathbb{P}\Big(X(t)=-ct\;\Big\lvert\,\mathcal{T}>t\Big)=\frac{1}{2}\,e^{-\lambda(1-p)t}.\label{probmassendpts}\end{equation}
Clearly, the endpoint masses in (\ref{probmassendpts}) coincide with those of a standard telegraph process having switching intensity $\lambda(1-p)$. This further confirms that the conditional law of $X(t)$ given ${\mathcal{T}>t}$ coincides with the law of a standard telegraph process with switching intensity $\lambda(1-p)$. Its absolutely continuous component can therefore be expressed explicitly in terms of Bessel functions, as proved by Orsingher \cite{orsingher}. For completeness, we provide the details of the derivation in the following theorem and subsequently discuss its probabilistic interpretation.
\begin{thm}
The probability density function $g(x,t)$ is
\begin{equation}g(x,t)=\frac{e^{-\lambda(1-p)t}}{2c}\Bigg[\lambda(1-p)I_0\left(\frac{\lambda(1-p)}{c}\sqrt{c^2t^2-x^2}\right)+\frac{\partial}{\partial t}I_0\left(\frac{\lambda(1-p)}{c}\sqrt{c^2t^2-x^2}\right)\Bigg],\qquad \lvert x\lvert<ct.\label{manglumron}\end{equation}
\end{thm}
\begin{proof}
By defining the auxiliary function
\begin{equation}\widetilde{g}(x,t)=e^{\lambda(1-p)t}g(x,t)\label{botsw}\end{equation}
and taking into account equation (\ref{gpde}), it can be verified that $\widetilde{g}(x,t)$ satisfies the partial differential equation
\begin{equation}\frac{\partial^2\widetilde{g}}{\partial t^2}-c^2\frac{\partial^2\widetilde{g}}{\partial x^2}=\lambda^2(1-p)^2\widetilde{g}.\label{gtildeeq}\end{equation}
The change of variables
\begin{equation}z=\sqrt{c^2t^2-x^2}\label{besseltrans}\end{equation}
now transforms equation (\ref{gtildeeq}) into the Bessel-type equation
$$\frac{d^2\widetilde{g}}{dz^2}+\frac{1}{z}\frac{d\widetilde{g}}{dz}-\frac{\lambda^2(1-p)^2}{c^2}\widetilde{g}=0,$$
whose general solution is
$$\widetilde{g}(z)=A\, I_0\left(\frac{\lambda(1-p)}{c}z\right)+B\, K_0\left(\frac{\lambda(1-p)}{c}z\right).$$
We disregard the term involving the modified Bessel function of the second kind $K_0(z)$, which diverges at $z=0$. Since equation (\ref{gpde}) is linear and has constant coefficients, the time derivative of any of its solutions is again a solution. Therefore, after inverting transformations (\ref{besseltrans}) and (\ref{botsw}), we consider a solution of the form
\begin{equation}g(x,t)=e^{-\lambda(1-p)t}\left[A I_0\left(\frac{\lambda(1-p)}{c}\sqrt{c^2t^2-x^2}\right)+B\frac{\partial}{\partial t}I_0\left(\frac{\lambda(1-p)}{c}\sqrt{c^2t^2-x^2}\right)\right].\label{Besseleqdenstel}\end{equation}
In order to determine the coefficients $A$ and $B$, we recall that the support of the process $X(t)$ coincides with the interval $[-ct,ct]$, and the probability mass at the endpoints of this interval is given in formula (\ref{probmassendpts}). Hence, in the interior of the support, it holds that
\begin{equation}\int_{-ct}^{ct}g(x,t)\,dx=1-e^{-\lambda(1-p)t}.\label{intcodn}\end{equation}
In order to impose the condition (\ref{intcodn}), we use the identities
$$\int_{-ct}^{ct}I_0\left(\frac{\lambda(1-p)}{c}\sqrt{c^2t^2-x^2}\right)dx=\frac{2c}{\lambda(1-p)}\sinh\big(\lambda(1-p)t\big)$$
and
$$\int_{-ct}^{ct}\frac{\partial}{\partial t}I_0\left(\frac{\lambda(1-p)}{c}\sqrt{c^2t^2-x^2}\right)dx=2c\left[\cosh\big(\lambda(1-p)t\big)-1\right].$$
Therefore,
\begin{align}
\int_{-ct}^{ct}g(x,t)\,dx=e^{-\lambda(1-p)t}\Bigg\{&\frac{Ac}{\lambda(1-p)}\left[e^{\lambda(1-p)t}-e^{-\lambda(1-p)t}\right]\nonumber\\
&+Bc\left[e^{\lambda(1-p)t}+e^{-\lambda(1-p)t}-2\right]\Bigg\}.\label{allinecaz}
\end{align}
By comparing now formulas (\ref{intcodn}) and (\ref{allinecaz}), we immediately obtain that
$$A=\frac{\lambda(1-p)}{2c},\qquad B=\frac{1}{2c},$$
which completes the proof.
\end{proof}
Having proved that the conditional process $X(t)\big\lvert\{\mathcal{T}>t\}$ is a telegraph process with switching intensity $\lambda(1-p)$, we now turn to the probabilistic interpretation of this result. We emphasize that the unconditional process $X(t)$ is telegraph process with switching intensity $\lambda$, and that killing may occur each time a change of direction takes place. Thus, conditioning on the event ${\mathcal{T}>t}$ is equivalent to restricting the probability space to include only a subset of the Poisson events, namely those that do not result in the process $X(t)$ being killed. In particular, the conditioning leads to the exclusion of a proportion $p$ of the Poisson events, while retaining the remaining proportion $1-p$. This modification of a Poisson process, which involves discarding a proportion of its events, is known in the probability literature as thinning of the Poisson process (we refer the reader to Chapter 5 of the book by Last and Penrose \cite{lastpenrose}). The thinning theorem states that a thinned Poisson process is still a Poisson process, but with a lower intensity. In our model, conditioning on killing not having occurred is equivalent to thinning the Poisson process governing the direction changes. Since the thinned process is itself a Poisson process, the conditional process $X(t)\big\lvert\{\mathcal{T}>t\}$ remains a telegraph process, although its switching intensity is reduced to reflect the thinning. In particular, it can be proved that the direction changes of $X(t)\big\lvert\{\mathcal{T}>t\}$ are paced by the thinned Poisson process $N(t)\big\lvert\{\mathcal{T}>t\}$, whose probability mass function is $$\mathbb{P}\Big(N(t)=k\,\Big\lvert\,\mathcal{T}>t\Big)=e^{-\lambda(1-p)}\;\frac{\lambda^k(1-p)^k}{k!},\qquad k\in\mathbb{N}_0.$$
In view of formula (\ref{Besseleqdenstel}), we are also able to study the asymptotic behaviour of the conditional process $X(t)\big\lvert\{\mathcal{T}>t\}$ in the hydrodynamic limit for $\lambda,c\to+\infty$, $p\to0$, under the scaling $\frac{\lambda}{c^2}\to1,\,\lambda p\to1$. By using the method discussed by Orsingher \cite{orsingher} to study the hydrodynamic limit of the density function (\ref{Besseleqdenstel}), it can be proved that
$$\lim_{\substack{\lambda,c\to+\infty\\p\to0}}\mathbb{P}\Big(X(t)\in\mathop{dx}\big\lvert\,\mathcal{T}>t\Big)/dx=\lim_{\substack{\lambda,c\to+\infty\\p\to0}}\mathbb{P}\Big(X(t)\in\mathop{dx}\Big)/dx=\frac{1}{\sqrt{2\pi t}}\,e^{-\frac{x^2}{2t}}.$$ This identity further supports the claim made in Section \ref{sec:uniT} that $X(t)$ and $\mathcal{T}$ become independent in the hydrodynamic limit.

We continue our analysis by investigating the joint distribution of the position of the telegraph process at the killing time $X(\mathcal{T})$ and the killing time $\mathcal{T}$ itself. It is immediate to verify that the support of the random vector $\big(X(\mathcal{T}),\,\mathcal{T}\big)$ is the infinite cone $C=\{(x,t)\in\mathbb{R}^2:\;\lvert x\lvert \le ct,\;t>0\}$. Moreover, the distribution exhibits a singular component on the boundary of the cone $C$. In fact, it can be verified that
\begin{equation}\label{bennipc}\mathbb{P}\Big(\left\lvert X(\mathcal{T})\right\lvert=c\,\mathcal{T}\Big)=\mathbb{P}\Big(\mathcal{T}=T_1\Big)=p.\end{equation}
Formula (\ref{bennipc}) highlights that the boundary of the cone $C$ identifies the trajectories of the telegraph process which are killed at the first direction change. In order to determine the exact distribution of $\big(X(\mathcal{T}),\,\mathcal{T}\big)$ over the boundary of $C$, it is sufficient to observe that
\begin{equation}\mathbb{P}\Big(X(\mathcal{T})=ct,\,\mathcal{T}\in dt\Big) /dt=\mathbb{P}\Big(D(0)=d_0,\,\mathcal{T}=T_1,\,T_1\in dt\Big) /dt=\frac{\lambda\,p}{2}\,e^{-\lambda t},\qquad t>0.\label{benpcda}\end{equation}
As for the distribution of $\big(X(\mathcal{T}),\,\mathcal{T}\big)$ in the interior of the support $C$, we observe that, for a small time interval of amplitude $\Delta t$,
\begin{align}
\mathbb{P}\Big(X(\mathcal{T})\in dx,\;\mathcal{T}\in (t,\;t+\Delta t]\Big)&=\mathbb{P}\Big(X(t)\in dx,\;\mathcal{T}\in (t,\;t+\Delta t]\Big)+o(\Delta t)\nonumber\\
&=\mathbb{P}\Big(X(t)\in dx,\;\mathcal{T}>t,\;\mathcal{T}\in (t,\;t+\Delta t]\Big)+o(\Delta t)\nonumber\\
&=\mathbb{P}\Big(X(t)\in dx\Big\lvert \mathcal{T}>t\Big)\cdot \mathbb{P}\Big(\mathcal{T}\in (t,\;t+\Delta t]\Big\lvert \mathcal{T}>t\Big)\cdot\mathbb{P}\Big(\mathcal{T}>t\Big)+o(\Delta t)\nonumber\\
&=\lambda\,p\,\Delta t\;e^{-\lambda p t}\,\mathbb{P}\Big(X(t)\in dx\Big\lvert \mathcal{T}>t\Big)+o(\Delta t),\qquad \lvert x\lvert<ct,\, t>0.\label{coscopia}
\end{align}
Dividing formula (\ref{coscopia}) by $\Delta t$ and taking the limit for $\Delta t\to0$ yields
\begin{align*}
\frac{\mathbb{P}\Big(X(\mathcal{T})\in dx,\;\mathcal{T}\in dt\Big)}{dx\,dt}=\frac{\lambda p\,e^{-\lambda t}}{2c}\,\Bigg[\lambda(&1-p)I_0\left(\frac{\lambda(1-p)}{c}\sqrt{c^2t^2-x^2}\right)\\
&+\frac{\partial}{\partial t}I_0\left(\frac{\lambda(1-p)}{c}\sqrt{c^2t^2-x^2}\right)\Bigg],\qquad \lvert x\lvert<ct,\, t>0.
\end{align*}
By including the singular component (\ref{benpcda}), and denoting by $\delta(\cdot)$ the Dirac delta function, we can finally write that
\begin{align}
\mathbb{P}\Big(X(\mathcal{T})\in dx,\;\mathcal{T}\in dt\Big)=\,&\frac{\lambda p\,e^{-\lambda t}}{2}\big[\delta(x-ct)+\delta(x+ct)\big]\,dx\,dt\nonumber\\
&+\frac{\lambda p\,e^{-\lambda t}}{2c}\mathbbm{1}_{\{|x|<ct\}}\Bigg[\lambda(1-p)I_0\left(\frac{\lambda(1-p)}{c}\sqrt{c^2t^2-x^2}\right)\nonumber\\
&\qquad\qquad\qquad\qquad\qquad+\frac{\partial}{\partial t}I_0\left(\frac{\lambda(1-p)}{c}\sqrt{c^2t^2-x^2}\right)\Bigg]\,dx\,dt,\qquad (x,t)\in C.\label{biscfreddgz}
\end{align}
We observe that several consistency checks can be performed to verify that the distribution in (\ref{biscfreddgz}) agrees with the results obtained so far. In particular, integrating the density in $(\ref{biscfreddgz})$ with respect to $x$ and $t$ yields formulas (\ref{lpdens}) and (\ref{ffubit}), respectively. We leave the details of this verification to the reader. However, we emphasize that integration with respect to $x$ can be achieved by using the Sommerfeld formula (see Magnus and Oberhettinger \cite{magnus1943formeln}, p. 34), which states that, for $x,r\in\mathbb{R}$,
\begin{equation*}
\frac{e^{i\kappa\sqrt{r^2+x^2}}}{\sqrt{r^2+x^2}}=
\int_0^\infty J_0(\tau r)\,\frac{\tau\;e^{-|x|\sqrt{\tau^2-\kappa^2}}}{\sqrt{\tau^2-\kappa^2}}\,d\tau,
\qquad-\frac{\pi}{2}\le\arg(\sqrt{\tau^2-\kappa^2})<\frac{\pi}{2},\;0\le\arg(\kappa)<\pi.
\end{equation*}
\noindent At this point, we have all the ingredients needed to derive the explicit distribution of the killed telegraph process $X(t\wedge\mathcal{T})$. The resulting distribution is stated in the following theorem.
\begin{thm}For $t>0$ and $x\in[-ct,ct]$, it holds that
$$
\begin{aligned}
\mathbb{P}\Big(X(t\wedge\mathcal{T})\in dx\Big)/dx
={}&\frac{e^{-\lambda t}}{2}\big[\delta(x-ct)+\delta(x+ct)\big]\\
&+\frac{\mathbbm{1}_{\{|x|<ct\}}}{2c}\Bigg\{e^{-\lambda t}\left[\lambda I_0\left(\frac{\lambda(1-p)}{c}\sqrt{c^2t^2-x^2}\right)+\frac{\partial}{\partial t}I_0\left(\frac{\lambda(1-p)}{c}\sqrt{c^2t^2-x^2}\right)\right]\\
&\qquad+\lambda^2p(2-p)\int_{|x|/c}^{t}e^{-\lambda s}I_0\left(\frac{\lambda(1-p)}{c}\sqrt{c^2s^2-x^2}\right)\,ds\Bigg\}.
\end{aligned}
$$
\end{thm}
\begin{proof}
By conditioning on whether killing occurs before time $t$, we have
$$
\mathbb{P}\big(X(t\wedge\mathcal{T})\in dx\big)=\mathbb{P}\big(X(t)\in dx,\mathcal{T}>t\big)+\int_0^t\mathbb{P}\big(X(\mathcal{T})\in dx,\mathcal{T}\in ds\big).$$
The first term can be obtained by using the probability density function (\ref{manglumron}) and the related singular distribution (\ref{probmassendpts}), whereas the second term is obtained from $(\ref{biscfreddgz})$. Thus, proving the result is a matter of straightforward calculation.
\end{proof}

\section{On the planar case with an infinite number of directions}\label{sec:4}
So far, we have examined the effect of the killing mechanism on the distribution of a one-dimensional telegraph process. In this section, we extend our analysis to the planar case. Specifically, we consider the planar finite-velocity motion introduced by Orsingher and De Gregorio \cite{orsingherdegreg}. The authors introduced a planar random motion $\big(X(t),\;Y(t))$ starting at the origin at time $t=0$, with an initial direction chosen uniformly at random from $[0,2\pi)$. The process then moves with constant velocity $c$ and changes direction according to a Poisson process $N(t)$ with constant intensity $\lambda$. Each time a Poisson event occurs, a new direction is chosen randomly with uniform distribution. Denoting by $T_k,\;k\in\mathbb{N}_0$, the arrival times of the Poisson process, Orsingher and De Gregorio \cite{orsingherdegreg} observed that the random vector $\big(X(t),\;Y(t))$ at the arrival $T_n$ admits the explicit representation
$$
\big(X(T_n),Y(T_n)\big)
=
\left(
c\sum_{k=1}^{n}(T_k-T_{k-1})\cos\Theta_k,\,
c\sum_{k=1}^{n}(T_k-T_{k-1})\sin\Theta_k
\right).
$$
where $\Theta_k,\;k\in\mathbb{N}$ are independent and identically distributed random variables with continuous uniform distribution over the interval $[0,2\pi)$.
Similarly to the univariate case, however, we introduce a killing mechanism for this process by defining a killing time $\mathcal{T}$ with geometric distribution over the arrival times of $N(t)$. In particular, by assuming that the process may be killed with probability $p$ every time a direction change occurs, we have that
\begin{equation}\mathbb{P}\Big(\mathcal{T}=T_k\Big\lvert T_1,\,T_2,\,T_3,\,...\Big)=(1-p)^{k-1} p,\qquad k\in\mathbb{N}.\label{bivporc}\end{equation}
Our aim is to investigate the distribution of the killed bivariate process $\big(X(t\wedge\mathcal{T}),\;Y(t\wedge\mathcal{T}))$. Some sample paths of the planar motion subject to the killing mechanism are illustrated in Figure \ref{fig:planar_trajectories}.
\begin{figure}[htbp]
    \centering
    \includegraphics[width=0.4\textwidth]{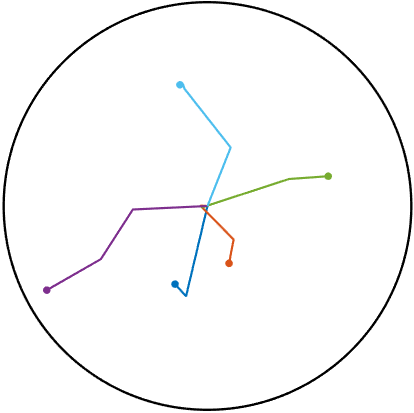}
    \caption{Five sample trajectories of the planar random motion with infinitely many directions and killing at Poisson events. The particle starts at the origin, moves at constant speed $c=25$, and changes direction at Poisson times with intensity $\lambda=3.5$, choosing each new direction uniformly on $[0,2\pi)$. At each Poisson event, killing occurs independently with probability $p=0.05$. Colored dots indicate the killing positions, while the black circle represents the boundary of the disk of radius $cT=25$, with $T=1$. The trajectories are sampled conditionally on killing occurring before time $T$.}
\label{fig:planar_trajectories}
\end{figure}
We begin our analysis by studying the distribution of the process at the killing time, namely the distribution of the random vector $\big(X(\mathcal{T}),Y(\mathcal{T})\big)$. In view of formula (\ref{bivporc}), we can write
\begin{equation}\big(X(\mathcal{T}),\;Y(\mathcal{T}))=\left(c\sum_{k=1}^{N(\mathcal{T})}(T_{k}-T_{k-1})\cos\Theta_k,\;c\sum_{k=1}^{N(\mathcal{T})}(T_{k}-T_{k-1})\,\sin\Theta_k \right)\label{bivporcdu}\end{equation}
Formula (\ref{bivporcdu}) permits us to derive the characteristic function $\big(X(\mathcal{T}),\;Y(\mathcal{T}))$. We start by studying the characteristic function conditional on the number of direction changes up to the killing time, which is given in the following theorem.
\begin{thm}\label{infcondchf}
For every $n\geq 1$, the characteristic function of $\big(X(\mathcal{T}),\;Y(\mathcal{T}))$ conditional on $N(\mathcal{T})=n$ is
\begin{equation}\mathbb{E}\left(e^{i[\alpha X(\mathcal{T})+\beta Y(\mathcal{T})]}\Big \lvert N(\mathcal{T})=n\right)=\left(\frac{\lambda}{\sqrt{\lambda^2+c^2(\alpha^2+\beta^2)}}\right)^n,\qquad \alpha,\beta\in\mathbb{R}.\label{infcondf}\end{equation}
\end{thm}
\begin{proof}Since the interarrival times $T_k-T_{k-1},\;k\ge1,$ are exponentially distributed, it holds that
\begin{align}
\mathbb{E}\left(e^{i[\alpha X(\mathcal{T})+\beta Y(\mathcal{T})]}\Big\lvert N(\mathcal{T})=n\right)
&=\mathbb{E}\left[\exp\left(i\alpha c\sum_{k=1}^{n}(T_k-T_{k-1})\cos\Theta_k+i\beta c\sum_{k=1}^{n}(T_k-T_{k-1})\sin(\Theta_k)\right)\right] \nonumber \\
&=\int_{0}^{\infty}ds_1\cdots\int_{0}^{\infty}ds_n\,\int_{0}^{2\pi}d\theta_1\cdots\int_{0}^{2\pi}d\theta_n\;\left(\frac{\lambda}{2\pi}\right)^n e^{-\lambda\sum_{k=1}^{n}s_k}\nonumber \\
&\qquad\qquad\qquad \cdot\exp\left(i\alpha c\sum_{k=1}^{n}s_k\cos\theta_k+i\beta c\sum_{k=1}^{n}s_k\sin\theta_k\right) \nonumber \\
&=\int_{0}^{\infty}ds_1\cdots\int_{0}^{\infty}ds_n\,\lambda^ne^{-\lambda\sum_{k=1}^{n}s_k}\prod_{k=1}^{n}\frac{1}{2\pi}\int_{0}^{2\pi}e^{i\alpha c s_k\cos\theta_k+i\beta c s_k\sin\theta_k}\,d\theta_k \nonumber \\
&=\int_{0}^{\infty}ds_1\cdots\int_{0}^{\infty}ds_n\,\lambda^ne^{-\lambda\sum_{k=1}^{n}s_k}\prod_{k=1}^{n}J_0\left(c s_k\sqrt{\alpha^2+\beta^2}\right) \nonumber
\end{align}
where we have used the following integral representation for the Bessel function (see Watson \cite{watson})
\begin{equation*}
J_0\left(x\sqrt{a^2+b^2}\right)=\frac{1}{2\pi}\int_{0}^{2\pi}\exp\left\{ix\left(a\cos\theta+b\sin\theta\right)\right\}\,d\theta,\qquad x\in\mathbb{R}.
\end{equation*}
By recalling now the Laplace transform of the Bessel function $J_0(\cdot)$, that is
\begin{equation*}
\int_{0}^{+\infty}e^{-y s}J_0(as)\,ds=\frac{1}{\sqrt{y^2+a^2}},\qquad a\in\mathbb{R},\,y>0
\end{equation*}
we can write that
\begin{align}
\mathbb{E}\left(e^{i[\alpha X(\mathcal{T})+\beta Y(\mathcal{T})]}\Big\lvert N(\mathcal{T})=n\right)&=\lambda^n\prod_{k=1}^{n}\int_{0}^{\infty}e^{-\lambda s_k}J_0\left(c s_k\sqrt{\alpha^2+\beta^2}\right)\,ds_k \nonumber \\
&=\lambda^n\prod_{k=1}^{n}\frac{1}{\sqrt{\lambda^2+c^2(\alpha^2+\beta^2)}} \nonumber \\
&=\left(\frac{\lambda}{\sqrt{\lambda^2+c^2(\alpha^2+\beta^2)}}\right)^n\nonumber
\end{align}
which completes the proof.
\end{proof}
At this point, we can take the expected value of the characteristic function (\ref{infcondf}) in order to remove the conditioning with respect to $N(\mathcal{T})$. Thus, we are able to obtain the unconditional characteristic function of the killed process $\big(X(\mathcal{T}),\;Y(\mathcal{T}))$ in the following theorem.

\begin{thm}\label{Laleinvac}
The characteristic function of $\big(X(\mathcal{T}),\;Y(\mathcal{T}))$ is
\begin{equation}\mathbb{E}\left(e^{i[\alpha X(\mathcal{T})+\beta Y(\mathcal{T})]}\right)=\frac{\lambda p}{\sqrt{\lambda^2+c^2(\alpha^2+\beta^2)}-\lambda(1-p)},\qquad \alpha,\beta\in\mathbb{R}.\label{infiunccf}\end{equation}
\end{thm}
\begin{proof}
By using Theorem \ref{infcondchf}, we can write that
\begin{align}
\mathbb{E}\left(e^{i[\alpha X(\mathcal{T})+\beta Y(\mathcal{T})]}\right)
&=\sum_{n=1}^{\infty}\mathbb{P}(N(\mathcal{T})=n)\;\mathbb{E}\left[e^{i[\alpha X(\mathcal{T})+\beta Y(\mathcal{T})]}\,\big\lvert\, N(\mathcal{T})=n\right] \nonumber \\
&=p\sum_{n=1}^{\infty}(1-p)^{n-1}\left(\frac{\lambda}{\sqrt{\lambda^2+c^2(\alpha^2+\beta^2)}}\right)^n \nonumber \\
&=\frac{p}{1-p}\sum_{n=1}^{\infty}\left(\frac{\lambda(1-p)}{\sqrt{\lambda^2+c^2(\alpha^2+\beta^2)}}\right)^n \nonumber\\
&=\frac{\lambda p}{\sqrt{\lambda^2+c^2(\alpha^2+\beta^2)}-\lambda(1-p)}.\nonumber
\end{align}
\end{proof}
So far, we have provided a characterization of the distribution of the killed bivariate motion through the corresponding characteristic function. However, we are interested in providing an explicit representation of the probability density function of the $\big(X(\mathcal{T}),\;Y(\mathcal{T}))$. To this aim, we define the density function
\begin{equation}f(x,y)\;dx\,dy=\mathbb{P}\Big(X(\mathcal{T})\in dx,\;Y(\mathcal{T})\in dy\Big)\label{inffdefuncond}\end{equation}
In order to study the density (\ref{inffdefuncond}), we first condition it to the number of direction changes $N(\mathcal{T})$. Thus, we consider the conditional probability density function
\begin{equation}f_n(x,y)\;dx\,dy=\mathbb{P}\Big(X(\mathcal{T})\in dx,\;Y(\mathcal{T})\in dy\;\Big\lvert\, N(\mathcal{T})=n\Big).\label{inffdefcond}\end{equation}
In the following theorem, we give an explicit representation of the conditional density (\ref{inffdefcond}) by calculating the inverse Fourier transform of the characteristic function (\ref{infcondf}).
\begin{thm}For every $n\geq 1$, the joint probability density function of $\big(X(\mathcal{T}),\;Y(\mathcal{T}))$ conditional on $N(\mathcal{T})=n$ is
\begin{equation}f_n(x,y)=\frac{1}{2^{\frac{n}{2}}\pi\;\Gamma\!\left(\frac n2\right)}\left(\frac{\lambda}{c}\right)^{\frac{n}{2}+1}\!(x^2+y^2)^{\frac{1}{2}\left(\frac{n}{2}-1\right)}\;K_{\frac{n}{2}-1}\!\left(\frac{\lambda}{c}\sqrt{x^2+y^2}\right)\label{lalalalellsls}\end{equation}
\end{thm}
\begin{proof}
By calculating the inverse Fourier transform of formula (\ref{infcondf}), we have that
\begin{align}
f_n(x,y)&=\frac{1}{(2\pi)^2}\int_{-\infty}^{+\infty}\int_{-\infty}^{+\infty}e^{-i\alpha x-i\beta y}\left(\frac{\lambda}{\sqrt{\lambda^2+c^2(\alpha^2+\beta^2)}}\right)^n\,d\alpha\,d\beta\nonumber\\
&=\frac{\lambda^n}{(2\pi)^2}\int_0^\infty\int_0^{2\pi}\frac{e^{-ir[x\cos\theta+y\sin\theta]}}{(\lambda^2+c^2r^2)^{n/2}}\,d\theta\,r\,dr\nonumber\\
&=\frac{\lambda^n}{2\pi}\int_0^\infty\frac{r\,J_0\!\left(r\sqrt{x^2+y^2}\right)}{(\lambda^2+c^2r^2)^{n/2}}\,dr.\nonumber
\end{align}
By using the identity (see Lebedev \cite{lebedev})
\begin{equation}
\int_0^\infty \frac{x^{\nu+1}J_\nu(bx)}{(x^2+a^2)^{\mu+1}}\,dx=\frac{a^{\nu-\mu}\,b^\mu}{2^\mu\;\Gamma(\mu+1)}K_{\nu-\mu}(ab),\qquad a>0,\,b>0,\;\; -1<\text{Re}(\nu)<2\,\text{Re}(\mu)+\frac{3}{2}\label{JKform}
\end{equation}
we can write that
\begin{align}
f_n(x,y)=\frac{1}{2^{\frac{n}{2}}\pi\;\Gamma\!\left(\frac n2\right)}\left(\frac{\lambda}{c}\right)^{\frac{n}{2}+1}\!(x^2+y^2)^{\frac{1}{2}\left(\frac{n}{2}-1\right)}\;K_{1-\frac{n}{2}}\!\left(\frac{\lambda}{c}\sqrt{x^2+y^2}\right).\nonumber
\end{align}
The proof is completed by recalling that $K_{\nu}(x)=K_{-\nu}(x)$.
\end{proof}
It is now interesting to observe that the probability density function (\ref{lalalalellsls}) is isotropic. Clearly, this implies that the distribution can be decomposed into its angular and radial independent components. While the angular component is uniformly distirbuted over $[0,2\pi)$ by construction, the radial component exhibits a nontrivial distribution. To derive it, we switch from the Cartesian coordinates $(x,y)$ to polar coordinates $(r,\theta)$ by means of the usual transformation
\begin{equation}\begin{cases}x=r\,\cos\theta\\y=r\,\sin\theta.\end{cases}\label{polarcoords}\end{equation}
By removing the angolar component from the density (\ref{lalalalellsls}), and taking into account the Jacobian of the transformation (\ref{polarcoords}), the radial component of the distribution is
\begin{equation}\widehat{f_n}(r):=2\pi r\,f_n(r\,\cos\theta,\,r\,\sin\theta)=\frac{1}{2^{\frac{n}{2}-1}\Gamma\!\left(\frac{n}{2}\right)}\left(\frac{\lambda}{c}\right)^{\frac{n}{2}+1}r^{\frac{n}{2}}K_{\frac{n}{2}-1}\!\left(\frac{\lambda}{c}r\right).\label{everyindimpad}\end{equation}
For the sake of completeness, we show that the radial density (\ref{everyindimpad}) integrates 1, and thus is a proper probability density function. For this purpose, we recall the Mellin transform of the Bessel function (Watson \cite{watson}, p. 388)
\begin{equation}\label{mellinklexver}\int_0^\infty x^{s-1}K_\nu(x)\,dx=2^{s-2}\Gamma\!\left(\frac{s-\nu}{2}\right)\Gamma\!\left(\frac{s+\nu}{2}\right),\qquad\operatorname{Re}(s)>|\operatorname{Re}(\nu)|,\qquad s,\nu\in\mathbb{C}.\end{equation}
This permits us to write that
\begin{align}\int_0^\infty \widehat{f_n}(r)\,dr&=\frac{1}{2^{\frac{n}{2}-1}\Gamma\!\left(\frac{n}{2}\right)}\left(\frac{\lambda}{c}\right)^{\frac{n}{2}+1}\int_0^\infty r^{\frac{n}{2}}K_{\frac{n}{2}-1}\!\left(\frac{\lambda}{c}r\right)\,dr \nonumber\\
&=\frac{1}{2^{\frac{n}{2}-1}\Gamma\!\left(\frac{n}{2}\right)}\left(\frac{\lambda}{c}\right)^{\frac{n}{2}+1}\left(\frac{c}{\lambda}\right)^{\frac{n}{2}+1}\int_0^\infty x^{\frac{n}{2}}K_{\frac{n}{2}-1}(x)\,dx \nonumber\\
&=\frac{1}{2^{\frac{n}{2}-1}\Gamma\!\left(\frac{n}{2}\right)}2^{\frac{n}{2}-1}\Gamma(1)\Gamma\!\left(\frac{n}{2}\right)=1\nonumber
\end{align} which yields the desired result.\\
\noindent Having obtained the exact distribution of the killed process $\bigl(X(\mathcal{T}),Y(\mathcal{T})\bigr)$ conditional on the number of Poisson events until the killing occurs, we now turn to the corresponding unconditional distribution. In principle, this distribution can be obtained either by averaging the conditional density $f_n(x,y)$ with respect to $N(\mathcal{T})$, or by computing the inverse Fourier transform of the characteristic function in \eqref{infiunccf}. As we will see, both approaches lead to the same result. Although an elementary closed-form expression for the unconditional density appears to be unavailable, we provide a representation based on an expansion in orthogonal polynomials. In particular, for $n\in\mathbb{N}_0$, we consider the Bessel polynomials (see Krall and Frink \cite{krallfrink})
\begin{equation}y_n(x)=\sum_{k=0}^n\frac{(n+k)!}{(n-k)!\;k!}\left(\frac{x}{2}\right)^k,\qquad x\in\mathbb{R}.\label{tanabespoldef}\end{equation}
It is well known that the Bessel polynomials are strictly related to the Bessel functions. Indeed, we recall that the Bessel function $K_{\nu}(\cdot)$ reduces to a polynomial in the case in which $\nu$ is a half-integer (see Watson \cite{watson}, p. 80). In particular, for \(n\in\mathbb{N}\), it holds that
\begin{align}
K_{n-\frac{1}{2}}(x)
&= \sqrt{\frac{\pi}{2x}}\,e^{-x}\sum_{k=0}^{n-1}\frac{(n-1+k)!}{(n-1-k)!\;k!}\frac{1}{(2x)^k}.\label{tanabesfundef}
\end{align}
By comparing formulas (\ref{tanabespoldef}) and (\ref{tanabesfundef}), it immediately follows that, for $n\in\mathbb{N}_0$,
\begin{equation}y_n(x)=\sqrt{\frac{2}{\pi x}}\;e^{\frac{1}{x}}\,K_{n+\frac{1}{2}}\left(\frac{1}{x}\right).\label{papalameri}\end{equation}
Thus, we are able to prove the following result.

\begin{thm}\label{thm:tempestfg}For $(x,y)\in\mathbb{R}\setminus\{(0,0)\}$, it holds that
\begin{align}
f(x,y)=&\frac{\lambda^2p(1-p)}{2\pi c^2}\;K_0\!\left(\frac{\lambda}{c}\sqrt{p(2-p)}\,\sqrt{x^2+y^2}\right)\nonumber\\
&\quad+ \frac{\lambda p}{2\pi c}\frac{e^{-\frac{\lambda}{c}\sqrt{x^2+y^2}}}{\sqrt{x^2+y^2}}
\left[
1+\frac{1}{2}\sum_{n=0}^{\infty}
\left(\frac{2\lambda(1-p)^2\sqrt{x^2+y^2}}{c}\right)^{n+1}
\!\!\frac{n!}{(2n+1)!}
\;y_n\!\left(\frac{c}{\lambda\sqrt{x^2+y^2}}\right)
\right]. \label{mutubenult}
\end{align}
\end{thm}
\begin{proof}
By recalling formulas (\ref{inffdefuncond}) and (\ref{inffdefcond}), and taking into account that $N(\mathcal{T})$ is geometrically distributed, we can calculate the density function $f(x,y)$ as
\begin{equation}f(x,y)=\sum_{n=1}^{\infty}p(1-p)^{n-1}f_n(x,y).\label{wdsum}\end{equation}
We treat the even and odd terms of the sum in (\ref{wdsum}) separately. In particular, by defining
$$f_{\mathrm{even}}(x,y)=\sum_{n=1}^{\infty}p(1-p)^{2n-1}f_{2n}(x,y),\qquad\qquad f_{\mathrm{odd}}(x,y)=\sum_{n=0}^{\infty}p(1-p)^{2n}f_{2n+1}(x,y)$$
it holds that
$$f(x,y)=f_{\mathrm{even}}(x,y)+f_{\mathrm{odd}}(x,y).$$
We start by analyzing the even terms of the sum. We have that
\begin{align}
f_{\mathrm{even}}(x,y)&=\frac{\lambda p}{\pi(1-p)c}\sum_{n=1}^{\infty}\frac{(1-p)^{2n}}{2^n\;\Gamma(n)}\left(\frac{\lambda}{c}\right)^n(x^2+y^2)^{\frac{n-1}{2}}K_{n-1}\!\left(\frac{\lambda}{c}\sqrt{x^2+y^2}\right) \nonumber\\
%&=\frac{\lambda p}{\pi(1-p)c}\sum_{n=1}^{\infty}\left(\frac{\lambda(1-p)^2}{2c}\right)^n\frac{(x^2+y^2)^{\frac{n-1}{2}}}{(n-1)!}\;K_{n-1}\!\left(\frac{\lambda}{c}\sqrt{x^2+y^2}\right) \nonumber\\
&=\frac{\lambda^2p(1-p)}{2\pi c^2}\sum_{n=0}^{\infty}\frac{1}{n!}\left(\frac{\lambda(1-p)^2\sqrt{x^2+y^2}}{2c}\right)^n K_n\!\left(\frac{\lambda}{c}\sqrt{x^2+y^2}\right). \nonumber
\end{align}
By using the following integral representation of the modified Bessel function of the second kind $K_{\nu}(\cdot)$ (see Watson \cite{watson})
\begin{align}
K_\nu(z)
&=\frac{1}{2}\left(\frac{z}{2}\right)^\nu\int_0^\infty e^{-t-\frac{z^2}{4t}}t^{-\nu-1}\,dt,\qquad z>0,\ \nu\in\mathbb R.
\end{align}
we can write
\begin{align}
f_{\mathrm{even}}(x,y)
%&=\frac{\lambda^2p(1-p)}{2\pi c^2}\sum_{n=0}^{\infty}\left(\frac{\lambda(1-p)^2}{2c}\right)^n\frac{(x^2+y^2)^{n/2}}{n!}K_n\!\left(\frac{\lambda}{c}\sqrt{x^2+y^2}\right) \nonumber\\
&=\frac{\lambda^2p(1-p)}{2\pi c^2}\sum_{n=0}^{\infty}\left(\frac{\lambda(1-p)^2}{2c}\right)^n
\frac{(x^2+y^2)^{n/2}}{n!}\frac{1}{2}\left(\frac{\lambda\sqrt{x^2+y^2}}{2c}\right)^n
\int_0^\infty e^{-t-\frac{\lambda^2(x^2+y^2)}{4c^2t}}t^{-n-1}\,dt \nonumber\\
&=\frac{\lambda^2p(1-p)}{4\pi c^2}\int_0^\infty e^{-t-\frac{\lambda^2(x^2+y^2)}{4c^2t}}\frac{1}{t}
\sum_{n=0}^{\infty}\frac1{n!}\left(\frac{\lambda^2(1-p)^2(x^2+y^2)}{4c^2t}\right)^n\,dt \nonumber\\
&=\frac{\lambda^2p(1-p)}{4\pi c^2}\int_0^\infty \frac{1}{t}
\exp\left\{-t-\frac{\lambda^2(x^2+y^2)}{4c^2t}+\frac{\lambda^2(1-p)^2(x^2+y^2)}{4c^2t}\right\}\,dt \nonumber\\
&=\frac{\lambda^2p(1-p)}{4\pi c^2}\int_0^\infty \frac{1}{t}
\exp\left\{-t-\frac{\lambda^2p(2-p)(x^2+y^2)}{4c^2t}\right\}\,dt \nonumber\\
&=\frac{\lambda^2p(1-p)}{2\pi c^2}\;K_0\!\left(\frac{\lambda}{c}\sqrt{p(2-p)}\,\sqrt{x^2+y^2}\right).\nonumber
\end{align}
As for the odd-order terms, by definition of $f_{\mathrm{odd}}(x,y)$, we can write that
\begin{align}
f_{\mathrm{odd}}(x,y)
&= \sum_{n=0}^{\infty} p(1-p)^{2n}\frac{1}{2^{n+\frac{1}{2}}\pi}\left(\frac{\lambda}{c}\right)^{n+\frac32}
\frac{(x^2+y^2)^{\frac{1}{2}(n-\frac{1}{2})}}{\Gamma\left(n+\frac{1}{2}\right)}
K_{n-\frac{1}{2}}\!\left(\frac{\lambda}{c}\sqrt{x^2+y^2}\right). \nonumber
\end{align}

By separating the term \(n=0\) from the remaining terms, and recalling that
\begin{equation}
K_{-\frac{1}{2}}(z)=K_{\frac{1}{2}}(z)= \sqrt{\frac{\pi}{2z}}\,e^{-z}\label{Khalfpm}
\end{equation}
we can write that
\begin{align}
f_{\mathrm{odd}}(x,y)
&= \frac{\lambda p}{2\pi c}\frac{e^{-\frac{\lambda}{c}\sqrt{x^2+y^2}}}{\sqrt{x^2+y^2}}+\sum_{n=1}^{\infty} p(1-p)^{2n}\frac{1}{2^{n+\frac{1}{2}}\pi}\left(\frac{\lambda}{c}\right)^{n+\frac{3}{2}}
\frac{(x^2+y^2)^{\frac{1}{2}(n-\frac{1}{2})}}{\Gamma\left(n+\frac{1}{2}\right)}
K_{n-\frac{1}{2}}\!\left(\frac{\lambda}{c}\sqrt{x^2+y^2}\right).\nonumber
\end{align}
The proof is complete by using the gamma duplication formula
$$\Gamma\left(n+\frac{1}{2}\right)= 2^{1-2n}\sqrt{\pi}\,\frac{\Gamma(2n)}{\Gamma(n)}$$ and the identity (\ref{papalameri}).
\end{proof}
In Theorem \ref{thm:tempestfg}, we have finally obtained an explicit expression for the probability density function of the killed process $\big(X(\mathcal{T}),\;Y(\mathcal{T}))$. This density was obtained by first conditioning the density on the number of Poisson events $N(\mathcal{T})$ and then averaging to remove the conditioning. However, in principle, the unconditional density can also be obtained by computing the inverse Fourier transform of the unconditional characteristic function (\ref{infiunccf}). In the next theorem, we follow this alternative approach to provide a verification of the density obtained in Theorem \ref{thm:tempestfg}.
\begin{thm}The inverse Fourier transform of (\ref{infiunccf}) coincides with the probability density function (\ref{mutubenult}).\end{thm}
\begin{proof}Denoting by $f(x,y)$ the inverse Fourier transform of (\ref{infiunccf}), it holds that
\begin{align}
f(x,y) &= \frac{1}{(2\pi)^2}\int_{-\infty}^{+\infty}\int_{-\infty}^{+\infty} e^{-i\alpha x-i\beta y}\frac{\lambda p}{\sqrt{\lambda^2+c^2(\alpha^2+\beta^2)}-\lambda(1-p)}\,d\alpha\,d\beta \nonumber\\
&= \frac{\lambda p}{(2\pi)^2}\int_0^{+\infty}\int_0^{2\pi}\frac{e^{-ixr\cos\theta-iyr\sin\theta}}{\sqrt{\lambda^2+c^2r^2}-\lambda(1-p)}\,r\,d\theta\,dr= \frac{\lambda p}{2\pi}\int_0^{+\infty}\frac{J_0\left(r\sqrt{x^2+y^2}\right)\,r}{\sqrt{\lambda^2+c^2r^2}-\lambda(1-p)}\,dr \nonumber\\
&= \frac{\lambda p}{2\pi}\int_0^{+\infty}\frac{rJ_0\left(r\sqrt{x^2+y^2}\right)\left[\sqrt{\lambda^2+c^2r^2}+\lambda(1-p)\right]}{c^2r^2+\lambda^2p(2-p)}\,dr \nonumber\\
&= \frac{\lambda^2p(1-p)}{2\pi}\int_0^{+\infty}\frac{rJ_0\left(r\sqrt{x^2+y^2}\right)}{c^2r^2+\lambda^2p(2-p)}\,dr+ \frac{\lambda p}{2\pi}\int_0^{+\infty}\frac{r\sqrt{\lambda^2+c^2r^2}\,J_0\left(r\sqrt{x^2+y^2}\right)}{c^2r^2+\lambda^2p(2-p)}\,dr \nonumber\\
&= \frac{\lambda^2p(1-p)}{2\pi c^2}K_0\left(\frac{\lambda\sqrt{p(2-p)}\sqrt{x^2+y^2}}{c}\right) + \frac{\lambda p}{2\pi}\int_0^{+\infty}\frac{r\,\sqrt{\lambda^2+c^2r^2}\,J_0\left(r\sqrt{x^2+y^2}\right)}{c^2r^2+\lambda^2p(2-p)}\,dr \nonumber\\
&= \frac{\lambda^2p(1-p)}{2\pi c^2}K_0\left(\frac{\lambda\sqrt{p(2-p)}\sqrt{x^2+y^2}}{c}\right)+ \frac{\lambda p}{2\pi}\int_0^{+\infty}\frac{r\,J_0\left(r\sqrt{x^2+y^2}\right)}{\sqrt{c^2r^2+\lambda^2}}\sum_{k=0}^{\infty}\left(\frac{\lambda^2(1-p)^2}{\lambda^2+c^2r^2}\right)^k\,dr.\nonumber
\end{align}
By reversing the order of summation and integration and using formula (\ref{JKform}), we finally obtain that
\begin{align}
f(x,y) &= \frac{\lambda^2p(1-p)}{2\pi c^2}K_0\left(\frac{\lambda\sqrt{p(2-p)}\sqrt{x^2+y^2}}{c}\right)+ \frac{\lambda p}{2\pi}\sum_{n=0}^{\infty}\left(\lambda^2(1-p)^2\right)^n\int_0^{+\infty}\frac{r\,J_0\left(r\sqrt{x^2+y^2}\right)}{(c^2r^2+\lambda^2)^{n+\frac{1}{2}}}\,dr\nonumber\\
&= \frac{\lambda^2p(1-p)}{2\pi c^2}K_0\left(\frac{\lambda\sqrt{p(2-p)}\sqrt{x^2+y^2}}{c}\right)\nonumber\\
&\qquad\qquad+ \frac{\lambda p}{2\pi c}\sqrt{\frac{2\lambda}{c\sqrt{x^2+y^2}}}\sum_{n=0}^{\infty}\left(\frac{\lambda(1-p)^2\sqrt{x^2+y^2}}{2c}\right)^{\!\!n}\; \frac{K_{n-\frac{1}{2}}\left(\frac{\lambda}{c}\sqrt{x^2+y^2}\right)}{\Gamma\left(n+\frac{1}{2}\right)}.\label{traum}
\end{align}
Isolating the first term of the infinite series in (\ref{traum}) yields
\begin{align}
f(x,y) &= \frac{\lambda^2p(1-p)}{2\pi c^2}K_0\left(\frac{\lambda\sqrt{p(2-p)}\sqrt{x^2+y^2}}{c}\right)\nonumber\\
&\qquad\qquad+ \frac{\lambda p}{2\pi c}\sqrt{\frac{2\lambda}{c\sqrt{x^2+y^2}}}\Bigg[\frac{1}{\sqrt{\pi}}\,K_{-\frac{1}{2}}\left(\frac{\lambda}{c}\sqrt{x^2+y^2}\right)\nonumber\\
&\qquad\qquad\qquad\qquad\qquad\qquad+\sum_{n=0}^{\infty}\left(\frac{\lambda(1-p)^2\sqrt{x^2+y^2}}{2c}\right)^{\!\!n+1}\; \frac{K_{n+\frac{1}{2}}\left(\frac{\lambda}{c}\sqrt{x^2+y^2}\right)}{\Gamma\left(n+\frac{3}{2}\right)}\Bigg].\nonumber
\end{align}
The identities (\ref{papalameri}) and (\ref{Khalfpm}) now imply that
\begin{align}
f(x,y) &= \frac{\lambda^2p(1-p)}{2\pi c^2}K_0\left(\frac{\lambda\sqrt{p(2-p)}\sqrt{x^2+y^2}}{c}\right)\nonumber\\
&\qquad\qquad+ \frac{\lambda p}{2\pi c}\;\frac{e^{-\frac{\lambda}{c}\sqrt{x^2+y^2}}}{\sqrt{x^2+y^2}}\left[1+\sqrt{\pi}\sum_{n=0}^{\infty}\left(\frac{\lambda(1-p)^2\sqrt{x^2+y^2}}{2c}\right)^{\!\!n+1}\; \frac{y_n\left(\frac{c}{\lambda \sqrt{x^2+y^2}}\right)}{\Gamma\left(n+\frac{3}{2}\right)}\right].\nonumber
\end{align}
By applying the gamma duplication formula, we finally obtain
\begin{align}
f(x,y) &= \frac{\lambda^2p(1-p)}{2\pi c^2}K_0\left(\frac{\lambda\sqrt{p(2-p)}\sqrt{x^2+y^2}}{c}\right)\nonumber\\
&\qquad\qquad+ \frac{\lambda p}{2\pi c}\;\frac{e^{-\frac{\lambda}{c}\sqrt{x^2+y^2}}}{\sqrt{x^2+y^2}}\left[1+\frac{1}{2}\sum_{n=0}^{\infty}\left(\frac{2\lambda(1-p)^2\sqrt{x^2+y^2}}{c}\right)^{\!\!n+1}\,\frac{n!}{(2n+1)!}\;y_n\left(\frac{c}{\lambda \sqrt{x^2+y^2}}\right)\right]\nonumber
\end{align}
which coincides with the probability density function (\ref{mutubenult}). This completes the proof.
\end{proof}
We have now obtained the probability density $f(x,y)$ of the planar motion at the killing time, i.e. the random vector $\big(X(\mathcal{T}),\;Y(\mathcal{T}))$. As in the conditional case (\ref{lalalalellsls}), the unconditional density is isotropic, allowing us to derive the corresponding radial density. As a consistency check on the results obtained so far, we verify that this radial density integrates to one. To this end, we again use the decomposition $f(x,y)=f_{\mathrm{even}}(x,y)+f_{\mathrm{odd}}(x,y)$ introduced in Theorem \ref{thm:tempestfg} and integrate the contributions of $f_{\mathrm{even}}$ and $f_{\mathrm{odd}}$ separately. By eliminating the angular component of the distribution, switching to polar coordinates and taking into account the Jacobian of the transformation, the radial components of $f_{\mathrm{even}}$ and $f_{\mathrm{odd}}$ are respectively
$$\widehat{f_{\mathrm{even}}}(r)=\frac{\lambda^2p(1-p)\;r}{c^2}\;K_0\!\left(\frac{\lambda}{c}\sqrt{p(2-p)}\,r\right)$$
and
$$\widehat{f_{\mathrm{odd}}}(r)=\frac{\lambda p}{c}\;e^{-\frac{\lambda}{c}r}\left[1+\frac{1}{2}\sum_{n=0}^{\infty}\left(\frac{2\lambda(1-p)^2}{c}\right)^{n+1}\frac{n!}{(2n+1)!}\;r^{n+1}\,y_n\left(\frac{c}{\lambda r}\right)\right]. $$
By integrating with respect to $r$, and taking into account formula (\ref{mellinklexver}), we obtain that
$$\int_0^\infty \widehat{f_{\mathrm{even}}}(r)\,dr=\frac{1-p}{2-p}$$
and
\begin{align}\int_0^\infty \widehat{f_{\mathrm{odd}}}(r)\,dr&=p+\frac{\lambda p}{2c}\sum_{n=0}^{\infty}\left(\frac{2\lambda(1-p)^2}{c}\right)^{n+1}\frac{n!}{(2n+1)!}\int_0^\infty e^{-\frac{\lambda}{c}r}r^{n+1}y_n\left(\frac{c}{\lambda r}\right)\,dr \nonumber\\
&=p+\frac{\lambda p}{2c}\sum_{n=0}^{\infty}\left(\frac{2\lambda(1-p)^2}{c}\right)^{n+1}\frac{n!}{(2n+1)!}\sum_{k=0}^n\frac{(n+k)!}{(n-k)!\,k!}\left(\frac{c}{2\lambda}\right)^k\int_0^\infty e^{-\frac{\lambda}{c}r}r^{n-k+1}\,dr \nonumber\\
&=p+\frac{\lambda p}{2c}\sum_{n=0}^{\infty}\left(\frac{2\lambda(1-p)^2}{c}\right)^{n+1}\frac{n!}{(2n+1)!}\left(\frac{c}{\lambda}\right)^{n+2}\sum_{k=0}^n\frac{(n+k)!}{k!}(n-k+1)\left(\frac{1}{2}\right)^k \nonumber\\
&=p+\frac{\lambda p}{2c}\sum_{n=0}^{\infty}\left(\frac{2\lambda(1-p)^2}{c}\right)^{n+1}\,\left(\frac{c}{\lambda}\right)^{n+2}\frac{1}{2^n} \nonumber\\
&=p+p\sum_{n=0}^{\infty}(1-p)^{2n+2}=\frac{1}{2-p}\nonumber
\end{align}
where we have used the identity
$$\sum_{k=0}^n\frac{(n+k)!}{k!}(n-k+1)\left(\frac12\right)^k=\frac{(2n+1)!}{2^n\;n!}.$$
Denoting by $\widehat{f}(r)$ the radial component of $f(x,y)$, we are finally able to verify that
$$\int_0^\infty \widehat{f}(r)\,dr=\int_0^\infty \widehat{f_{\mathrm{even}}}(r)\,dr+\int_0^\infty \widehat{f_{\mathrm{odd}}}(r)\,dr=1.$$

We have now completed our analysis of the distribution of the planar motion with infinitely many directions at the killing time. Following the approach adopted in the one-dimensional case, we conclude our study of the planar model by investigating its distribution conditional on survival up to time $t$. For $x^2+y^2<c^2t^2$, we define the probability density function of the continuous component of the distribution as
$$g(x,y,t)\,dx\,dy=\mathbb{P}\bigl(X(t)\in dx,\;Y(t)\in dy\mid\mathcal{T}>t\bigr),$$
and its counterpart conditional on the number of Poisson events,
$$g_n(x,y,t)\,dx\,dy=\mathbb{P}\bigl(X(t)\in dx,\;Y(t)\in dy\mid N(t)=n,\;\mathcal{T}>t\bigr),\qquad n\geq1.$$
We begin with the conditional density $g_n$. By the same thinning argument used in Section 2, conditional on survival up to time $t$, the direction changes are governed by a homogeneous Poisson process with intensity $\lambda(1-p)$. However, as shown by Orsingher and De Gregorio \cite{orsingherdegreg}, the distribution of the random motion conditional on the number of direction changes does not depend on the Poisson intensity. In particular, they proved that
$$g_n(x,y,t)=\frac{n}{2\pi(ct)^n}\left(c^2t^2-x^2-y^2\right)^{\frac n2-1},\qquad x^2+y^2<c^2t^2,\quad n\geq1.$$
For $n=0$, the distribution is uniform on the circumference $x^2+y^2=c^2t^2$. The density conditional only on survival is then obtained by averaging over the number of events of the thinned Poisson process. Thus, the continuous component of the distribution reads
$$
\begin{aligned}
g(x,y,t)&=e^{-\lambda(1-p)t}\sum_{n=1}^{\infty}\frac{\bigl(\lambda(1-p)t\bigr)^n}{n!}\,g_n(x,y,t)\\
&=\frac{\lambda(1-p)}{2\pi c}\frac{\exp\!\left\{-\lambda(1-p)t+\frac{\lambda(1-p)}{c}\sqrt{c^2t^2-x^2-y^2}\right\}}{\sqrt{c^2t^2-x^2-y^2}},\qquad x^2+y^2<c^2t^2.
\end{aligned}
$$
The remaining mass, equal to $e^{-\lambda(1-p)t}$, is uniformly distributed on the circumference $x^2+y^2=c^2t^2$.\\
Similarly to the univariate case, we are now interested in deriving the joint distribution of the position of the planar process at the killing time and the killing time itself. The support of the random vector $\big(X(\mathcal{T}),Y(\mathcal{T}),\mathcal{T}\big)$ is the infinite cone
$$C=\left\{(x,y,t)\in\mathbb{R}^3:\;x^2+y^2\leq c^2t^2,\;t>0\right\}.$$
As in the one-dimensional case, the distribution has a singular component on the boundary of $C$. Indeed,
$$\mathbb{P}\left(X^2(\mathcal{T})+Y^2(\mathcal{T})=c^2\mathcal{T}^2\right)=\mathbb{P}(\mathcal{T}=T_1)=p,$$
since the boundary identifies trajectories killed at the first direction change. Hence, the same conditioning argument used in the one-dimensional case gives
\begin{align}
\mathbb{P}\Big(X(\mathcal{T})\in dx,\,Y(\mathcal{T})\in dy,\,\mathcal{T}\in dt\Big)&=\lambda\,p\,e^{-\lambda p t}\,\mathbb{P}\Big(X(t)\in dx,\,Y(t)\in dy\Big\lvert \mathcal{T}>t\Big)\,dt\nonumber\\
=\Big[\frac{\lambda p\,e^{-\lambda t}}{\pi}\,\delta\left(x^2+y^2-c^2t^2\right)&+\lambda p\,e^{-\lambda pt}\,g(x,y,t)\Big]dx\,dy\,dt,\qquad\qquad (x,y,t)\in C.\label{quazpsk}
\end{align}
At this point, the distribution of the killed planar process follows directly.
\begin{thm}
For $t>0$ and $x^2+y^2\leq c^2t^2$, it holds that
\begin{align}
\frac{\mathbb{P}\Big(X(t\wedge\mathcal{T})\in dx,Y(t\wedge\mathcal{T})\in dy\Big)}{dx\,dy}=&{}\frac{e^{-\lambda t}}{\pi}\delta\left(x^2+y^2-c^2t^2\right)\nonumber\\
+\mathbbm{1}_{\{x^2+y^2<c^2t^2\}}\Bigg[\frac{\lambda p}{2\pi c\,\sqrt{x^2+y^2}}&\exp\left\{-\frac{\lambda}{c}\sqrt{x^2+y^2}\right\}+e^{-\lambda pt}g(x,y,t)+\lambda p\int_{\sqrt{x^2+y^2}/c}^{t}e^{-\lambda ps}g(x,y,s)\,ds\Bigg].\nonumber\end{align}
\end{thm}
\begin{proof}
Conditioning on whether killing occurs before time $t$ gives
$$
\begin{aligned}
&\mathbb{P}\big(X(t\wedge\mathcal{T})\in dx,Y(t\wedge\mathcal{T})\in dy\big)\\
&\qquad=
\mathbb{P}\big(X(t)\in dx,Y(t)\in dy,\mathcal{T}>t\big)
+
\int_0^t
\mathbb{P}\big(X(\mathcal{T})\in dx,Y(\mathcal{T})\in dy,\mathcal{T}\in ds\big).
\end{aligned}
$$
The result follows by using formula (\ref{quazpsk}).
\end{proof}

\appendix
\section{On the sum of independent Laplace random variables}\label{appendixLaplace}
In Section \ref{sec:uniT}, we made use of the distribution of the sum of independent Laplace random variables. The probability density function of this sum, given in formula (\ref{intro:sumlapl}), was derived by Kotz et al. \cite{kotz} by exploiting properties of modified Bessel functions. In this appendix, we present an alternative derivation based on the inversion of the bilateral Laplace transform.\\
For $n\in\mathbb{N}$, let $L_1,\ldots,L_n$ be independent and identically distributed random variables with density
$$
h_1(x)=\frac{b}{2}e^{-b|x|},\qquad b>0,\quad x\in\mathbb{R}.
$$
The moment generating function of each $L_k$ is
$$
\mathbb{E}\left[e^{\alpha L_k}\right]=\frac{b^2}{b^2-\alpha^2},\qquad |\alpha|<b.
$$
Therefore, by independence of the Laplace random variables, if
$$
S_n=\sum_{k=1}^nL_k,
$$
then
$$
\mathbb{E}\left[e^{\alpha S_n}\right]=\frac{b^{2n}}{(b^2-\alpha^2)^n},\qquad |\alpha|<b.
$$
Let $h_n$ denote the probability density function of $S_n$. By inversion of the bilateral Laplace transform,
$$
h_n(x)=\frac{b^{2n}}{2\pi i}\int_{\gamma-i\infty}^{\gamma+i\infty}\frac{e^{-\alpha x}}{(b^2-\alpha^2)^n}\,d\alpha,\qquad \gamma\in(-b,b).
$$
The integrand has two poles of order $n$ at $\alpha=\pm b$. The half-plane in which the integration contour must be closed depends on the sign of $x$, since the exponential term $e^{-\alpha x}$ must vanish as $|\alpha|\to\infty$.
For $x>0$, the contour is closed in the right half-plane and the only enclosed pole is $\alpha=b$. Therefore, by computing the $n$-th order residue, we have that
\begin{align*}
h_n(x)
&=-\frac{1}{(n-1)!}\lim_{\alpha\to b}\frac{d^{n-1}}{d\alpha^{n-1}}\left[(\alpha-b)^n\frac{b^{2n}e^{-\alpha x}}{(b^2-\alpha^2)^n}\right]\\
&=-\frac{b^{2n}(-1)^n}{(n-1)!}\lim_{\alpha\to b}\frac{d^{n-1}}{d\alpha^{n-1}}\left[e^{-\alpha x}(\alpha+b)^{-n}\right]\\
&=-\frac{b^{2n}(-1)^n}{(n-1)!}\lim_{\alpha\to b}\sum_{k=0}^{n-1}\binom{n-1}{k}(-x)^{n-1-k}e^{-\alpha x}(-1)^k\frac{(n-1+k)!}{(n-1)!}(\alpha+b)^{-n-k}\\
&=\left(\frac{b}{2}\right)^n\frac{e^{-bx}}{(n-1)!}\sum_{k=0}^{n-1}\frac{(n-1+k)!}{(n-1-k)!k!}\frac{x^{n-1-k}}{(2b)^k},\qquad x>0.
\end{align*}
For $x<0$, the contour is instead closed in the left half-plane and the only enclosed pole is $\alpha=-b$. Hence
\begin{align*}
h_n(x)
&=\frac{1}{(n-1)!}\lim_{\alpha\to-b}\frac{d^{n-1}}{d\alpha^{n-1}}\left[(\alpha+b)^n\frac{b^{2n}e^{-\alpha x}}{(b^2-\alpha^2)^n}\right]\\
&=\frac{b^{2n}}{(n-1)!}\lim_{\alpha\to-b}\frac{d^{n-1}}{d\alpha^{n-1}}\left[e^{-\alpha x}(b-\alpha)^{-n}\right]\\
&=\frac{b^{2n}}{(n-1)!}\lim_{\alpha\to-b}\sum_{k=0}^{n-1}\binom{n-1}{k}(-x)^{n-1-k}e^{-\alpha x}\frac{(n-1+k)!}{(n-1)!}(b-\alpha)^{-n-k}\\
&=\left(\frac{b}{2}\right)^n\frac{e^{bx}}{(n-1)!}\sum_{k=0}^{n-1}\frac{(n-1+k)!}{(n-1-k)!k!}\frac{(-x)^{n-1-k}}{(2b)^k},\qquad x<0.
\end{align*}
Combining the two cases, and extending the expression to $x=0$ by continuity, we finally obtain
$$
h_n(x)=\left(\frac{b}{2}\right)^n\frac{e^{-b|x|}}{(n-1)!}\sum_{k=0}^{n-1}\frac{(n-1+k)!}{(n-1-k)!k!}\frac{|x|^{n-1-k}}{(2b)^k},\qquad x\in\mathbb{R}.
$$

\bibliographystyle{plain}
\nocite{*}
\bibliography{bibliography}

@book{kotz,
  title={The Laplace distribution and generalizations: a revisit with applications to communications, economics, engineering, and finance},
  author={Kotz, Samuel and Kozubowski, Tomasz and Podgorski, Krzystof},
  year={2012},
  publisher={Springer Science \& Business Media}}

@book{lebedev,
  title={Special functions and their applications},
  author={Lebedev, Nikola{\u\i}},
  year={1965},
  publisher={Prentice-Hall inc.},
  address={London}}

@book{watson,
  title     = {{A Treatise on the Theory of Bessel Functions}},
  author    = {Watson, George Neville},
  edition   = {2},
  publisher = {Cambridge University Press},
  address   = {Cambridge},
  year      = {1995}}

@article{krallfrink,
  title={{A new class of orthogonal polynomials: The Bessel polynomials}},
  author={Krall, Harry L and Frink, Orrin},
  journal={Transactions of the American Mathematical Society},
  volume={65},
  number={1},
  pages={100--115},
  year={1949}}

@article{orsingher,
  title={Probability law, flow function, maximum distribution of wave-governed random motions and their connections with {K}irchoff's laws},
  author={Orsingher, Enzo},
  journal={Stochastic Processes and their Applications},
  volume={34},
  number={1},
  pages={49--66},
  year={1990}
}

@book{lastpenrose,
  title={Lectures on the Poisson process},
  author={Last, G{\"u}nter and Penrose, Mathew},
  volume={7},
  year={2018},
  publisher={Cambridge University Press}
}

@article{orsingherdegreg,
  title={Random flights in higher spaces},
  author={Orsingher, Enzo and De Gregorio, Alessandro},
  journal={Journal of Theoretical Probability},
  volume={20},
  number={4},
  pages={769--806},
  year={2007}
}

@article{iuliano,
  title={On the Telegraph Process Driven by Geometric Counting Process with Poisson-Based Resetting: AD Crescenzo et al.},
  author={Di Crescenzo, Antonio and Iuliano, Antonella and Mustaro, Verdiana and Verasani, Gabriella},
  journal={Journal of Statistical Physics},
  volume={190},
  number={12},
  pages={191},
  year={2023}
}

@article{martinucci,
  title={Some results on the telegraph process driven by gamma components},
  author={Martinucci, Barbara and Meoli, Alessandra and Zacks, Shelemyahu},
  journal={Advances in Applied Probability},
  volume={54},
  number={3},
  pages={808--848},
  year={2022}
}

@article{evans,
  title={Run and tumble particle under resetting: a renewal approach},
  author={Evans, Martin R and Majumdar, Satya N},
  journal={Journal of Physics A: Mathematical and Theoretical},
  volume={51},
  number={47},
  pages={475003},
  year={2018}
}

@article{ratanov,
  title={Telegraph processes with random jumps and complete market models},
  author={Ratanov, Nikita},
  journal={Methodology and Computing in Applied Probability},
  volume={17},
  number={3},
  pages={677--695},
  year={2015}
}

@article{goldstein,
  title={On diffusion by discontinuous movements, and on the telegraph equation},
  author={Goldstein, Sidney},
  journal={The Quarterly Journal of Mechanics and Applied Mathematics},
  volume={4},
  number={2},
  pages={129--156},
  year={1951}
}

@article{kac,
  title={A stochastic model related to the telegrapher's equation},
  author={Kac, Mark},
  journal={The Rocky Mountain Journal of Mathematics},
  volume={4},
  number={3},
  pages={497--509},
  year={1974}
}

@article{cane,
  title={Diffusion models with relativity effects},
  author={Cane, Violet R},
  journal={Journal of Applied Probability},
  volume={12},
  number={S1},
  pages={263--273},
  year={1975}
}

@article{beghin2001probabilistic,
  title={Probabilistic analysis of the telegrapher' s process with drift by means of relativistic transformations},
  author={Beghin, Luisa and Nieddu, Luciano and Orsingher, Enzo},
  journal={International Journal of Stochastic Analysis},
  volume={14},
  number={1},
  pages={11--25},
  year={2001}
}

@article{orsingher2000exact,
  title={Exact joint distribution in a model of planar random motion},
  author={Orsingher, Enzo},
  journal={Stochastics: An International Journal of Probability and Stochastic Processes},
  volume={69},
  number={1-2},
  pages={1--10},
  year={2000}
}

@article{dicrminimal,
  title={Exact transient analysis of a planar random motion with three directions},
  author={Di Crescenzo, Antonio},
  journal={Stochastics: An International Journal of Probability and Stochastic Processes},
  volume={72},
  number={3-4},
  pages={175--189},
  year={2002}
}

@article{lachal,
  title={Cyclic random motions in $\mathbb{R}^d$-space with $n$ directions},
  author={Lachal, Aim{\'e}},
  journal={ESAIM: Probability and Statistics},
  volume={10},
  pages={277--316},
  year={2006},
  publisher={EDP Sciences}
}

@article{vortex,
  title={Planar random motions in a vortex},
  author={Orsingher, Enzo and Marchione, Manfred Marvin},
  journal={Journal of Theoretical Probability},
  volume={38},
  number={1},
  pages={4},
  year={2025}
}

@article{masoliver,
  author = {Masoliver, Jaume and Weiss, George H.},
  title = {Telegrapher's equations with variable propagation speeds},
  journal = {Physical Review E},
  volume = {49},
  number = {5},
  pages = {3852--3854},
  year = {1994}
}

@article{angelanietal,
  author = {Angelani, L. and {Di Leonardo}, R. and Paoluzzi, M.},
  title = {First-passage time of run-and-tumble particles},
  journal = {The European Physical Journal E},
  volume = {37},
  pages = {59},
  year = {2014}
}

@article{OMSV,
  author = {Orsingher, Enzo and Marchione, Manfred Marvin},
  title = {One-dimensional and planar random motions with variable propagation speeds},
  journal = {Journal of Theoretical Probability},
  volume = {39},
  number = {2},
  pages = {36},
  year = {2026}
}

@book{magnus1943formeln,
  title={Formeln und S{\"a}tze f{\"u}r die speziellen Funktionen der mathematischen Physik},
  author={Magnus, Wilhelm and Oberhettinger, Fritz},
  year={1943},
  publisher={Springer},
  address={Berlin, Heidelberg}
}
\end{document}